\documentclass[10pt]{amsart}
\usepackage{array,amsmath, enumerate, color, url}
\usepackage{amssymb}
\usepackage{graphicx,subfigure}
\makeatletter
 \usepackage{fullpage} % reduces the margins to 1 inch
 \usepackage{amsthm}
\usepackage{fullpage}

\makeatother

\begin{document}
\newcommand{\n}{\noindent}
\newtheorem{thm}{Theorem}[section]
\newtheorem{lem}[thm]{Lemma}
\newtheorem{prop}[thm]{Proposition}
\newtheorem{cor}[thm]{Corollary}
\newtheorem{con}[thm]{Conjecture}
\newtheorem{claim}[thm]{Claim}
\newtheorem{obs}[thm]{Observation}
\newtheorem{definition}[thm]{Definition}
\newtheorem{example}[thm]{Example}
\newtheorem{rmk}[thm]{Remark}
\newcommand{\di}{\displaystyle}
\def\dfc{\mathrm{def}}
\def\cF{{\cal F}}
\def\cH{{\cal H}}
\def\cT{{\cal T}}
\def\C{{\mathcal C}}
\def\cA{{\cal A}}
\def\cB{{\mathcal B}}
\def\P{{\mathcal P}}
\def\Q{{\mathcal Q}}
\def\cP{{\mathcal P}}
\def\cp{\alpha'}
\def\Frk{F_k^{2r+1}}

\title{Planar graphs without 5-cycles and intersecting triangles are $(1,1,0)$-colorable}
\author{Runrun Liu$^{\dagger}$ and Xiangwen Li$^{\dagger}$ and
Gexin Yu$^{\dagger}\ ^{\ddagger}$}\thanks{The second author was supported by the
Natural Science Foundation of China (11171129)  and by Doctoral Fund of Ministry of Education of China (20130144110001);
 The third author's research was supported in part by NSA grant H98230-12-1-0226.}

\address{$^{\dagger}$ Department of Mathematics, Huazhong Normal University, Wuhan, 430079, China}
\email{xwli68@mail.ccnu.edu.cn}
\address{$^{\ddagger}$ Department of Mathematics, The College of William and Mary,
Williamsburg, VA, 23185, USA.  }

\email{gyu@wm.edu}

\date{\today}
\maketitle

\begin{abstract}
A $(c_1,c_2,...,c_k)$-coloring of $G$ is a mapping $\varphi:V(G)\mapsto\{1,2,...,k\}$ such that for every $i,1 \leq i \leq k$, $G[V_i]$ has maximum degree at most $c_i$, where $G[V_i]$ denotes the subgraph induced by the vertices colored $i$. Borodin and Raspaud conjecture that every planar graph without $5$-cycles and intersecting triangles is $(0,0,0)$-colorable. We prove in this paper that such graphs are $(1,1,0)$-colorable.
\end{abstract}

\section{Introduction}
Graph coloring is one of the central topics in graph theory.   A graph is {\em $(c_1, c_2, \cdots, c_k)$-colorable} if the vertex set can be partitioned into $k$ sets $V_1,V_2, \ldots, V_k$, such that for every $i: 1\leq i\leq k$ the subgraph $G[V_i]$ has maximum degree at most $c_i$.    Thus a $(0,0,0)$-colorable graph is properly $3$-colorable.

The problem of deciding whether a planar graph is properly $3$-colorable is NP-complete. A lot of research has been devoted to finding conditions for a planar graph to be properly $3$-colorable. The well-known Gr\"{o}tzsch Theorem \cite{G59} shows that ``triangle-free'' suffices.  The famous Steinberg Conjecture \cite{S76} proposes that ``free of $4$-cycles and $5$-cycles'' is also enough.

\begin{con}[Steinberg, \cite{S76}]
All planar graphs without $4$-cycles and $5$-cycles are $3$-colorable.
\end{con}

Some relaxations of the Steinberg Conjecture are known to be true.  Along the direction suggested by Erd\H{o}s to find a constant $c$ such that a planar graph without cycles of length from $4$ to $c$ is $3$-colorable, Borodin, Glebov, Raspaud, and Salavatipour~\cite{BGRS05} showed that $c\le 7$, and more results similar to those can be found in the survey by Borodin~\cite{B12}.     Another direction of relaxation of the conjecture is to allow some defects in the color classes.   Chang, Havet, Montassier, and Raspaud~\cite{CHMR11} proved that all planar graphs without $4$-cycles or $5$-cycles are $(2,1,0)$-colorable and $(4,0,0)$-colorable.  In~\cite{HSWXY13, HY13, XMW12}, it is shown that planar graphs without $4$-cycles or $5$-cycles are $(3,0,0)$- and $(1,1,0)$-colorable.  Some more results along this directions can be found in the papers by  Wang {\it et al.} \cite{XMW12, XW13}.

Havel \cite{H69} proposed that planar graphs with triangles far apart should be properly $3$-colorable, which was confirmed in a recent preprint of Dvo\"r\'ak, Kr\'al and Thomas~\cite{DKT09}.  Borodin and Raspaud~\cite{BR03} combined the ideas of Havel and Steinberg and proposed the following so called  Bordeaux Conjecture in 2003.

\begin{con}[Borodin and Raspaud, \cite{BR03}] \label{con1}
Every planar graph without intersecting triangles and without $5$-cycles is $3$-colorable.
\end{con}

A planar graph without intersecting triangles means the distance between triangles is at least $1$.  Let $d^{\bigtriangledown}$ denote the smallest distance between any pair of triangles in a planar graph.  A relaxation of the Bordeaux Conjecture with $d^{\bigtriangledown}\ge 4$ was confirmed  by Borodin and Raspaud~\cite{BR03}, and the result was improved to $d^{\bigtriangledown}\ge 3$ by Borodin and Glebov  \cite{BG04} and, independently, by  Xu \cite{X07}.     Borodin and Glebov \cite{BG11} further improved the result to $d^\bigtriangledown\ge2$.

Using the relaxed coloring notation,  Xu~\cite{X08} proved that all planar graphs without adjacent triangles and  $5$-cycles are $(1,1,1)$-colorable, where two triangles are adjacent if they share an edge.

Let $\mathcal{G}$ be the family of plane graphs with $d^{\bigtriangledown}\ge 1$ and without $5$-cycles.   Yang and Yerger~\cite{YY14} showed that planar graphs in $\mathcal{G}$ are $(4,0,0)$- and $(2,1,0)$-colorable, but there is a flaw in one of their key lemmas (Lemma~2.4).  In \cite{LLY14}, we showed that graphs in $\mathcal{G}$ are $(2,0,0)$-colorable.

In this paper, we will prove another relaxation of the Bordeaux Conjecture.   Let $G$ be a graph and $H$ be a subgraph of $G$.  We call $(G,H)$ to be {\em superextendable} if each $(1,1,0)$-coloring of $H$ can be extended to $G$ so that vertices in $G-H$ have different colors from their neighbors in $H$;  in this case, we call $H$ to be a superextendable subgraph.

\begin{thm}\label{main}
Every triangle or $7$-cycle of a planar graph in $\mathcal{G}$ is superextendable.
\end{thm}

As a corollary, we have the following relaxation of the Bordeaux Conjecture.

\begin{thm}\label{main-1}
A planar graph in $\mathcal{G}$ is $(1,1,0)$-colorable.
\end{thm}

To see the truth of Theorem~\ref{main-1} by way of Theorem~\ref{main},  we may assume that the planar graph contains a triangle $C$ since $G$ is $(0,0,0)$-colorable if $G$ has no triangle. Then color the triangle, and by Theorem~\ref{main}, the coloring of $C$ can be superextended to $G$. Thus, we get a coloring of $G$.

As many results with similar fashion, we use a discharging argument to prove Theorem~\ref{main}.  This argument consists of two parts: structures and discharging. After introduce some common notations in Section~\ref{prelim},  we show in Section~\ref{reduce} some useful special structures in a minimal counterexample to the theorem, then in Section~\ref{discharging}, we design a discharging process to distribute the charges and use the special structures to reach a contradiction.

It should be noted that while the proof of our main theorem shares a lot of common properties with the $(2,0,0)$ result in \cite{LLY14}, it is much more involved.  We have to extend some powerful tools from \cite{X08} by Xu, and discuss in detail the structures around $4$-vertices and $5$-vertices.  It would be interesting to know how to use the new tools developed in this paper  to improve our result.

\section{Preliminaries}\label{prelim}
In this section, we introduce some notations used in the paper.

Graphs mentioned in this paper are all simple.  For a positive integer $n$, let $[n]=\{1,2,\ldots,n\}$.  A $k$-vertex ($k^+$-vertex, $k^-$-vertex) is a vertex of degree $k$ (at least $k$, at most $k$). The same notation will apply to faces and cycles. We use $b(f)$ to denote the vertex sets on $f$. We use $F(G)$ to denote the set of faces in $G$. An $(l_1, l_2, \ldots, l_k)$-face is a $k$-face $v_1v_2\ldots v_k$ with $d(v_i)=l_i$, respectively. A face $f$ is a {\em pendant $3$-face} of vertex $v$ if $v$ is not on $f$ but is adjacent to some $3$-vertex on $f$. A {\em pendant neighbor} of a $3$-vertex $v$ on a $3$-face is the neighbor of $v$ not on the $3$-face.

Let $C$ be a cycle of a plane graph $G$. We use $int(C)$ and $ext(C)$ to denote the sets of vertices located inside and outside $C$, respectively. The cycle $C$ is called a {\em separating cycle} if $int(C)\ne\emptyset\ne ext(C)$, and is called a {\em nonseparating cycle} otherwise. We still use $C$ to denote the set of vertices of $C$.

Let $S_1, S_2, \ldots, S_l$ be pairwise disjoint subsets of $V(G)$. We use $G[S_1, S_2, \ldots, ,S_l]$ to denote the graph obtained from $G$ by identifying all the vertices in $S_i$ to a single vertex for each $i\in\{1, 2, \ldots, l\}$.

A vertex $v$ is {\em properly colored} if all neighbors of $v$ have different colors from $v$. A vertex $v$ is {\em nicely colored} if it shares a color (say $i$) with at most max$\{s_i-1,0\}$ neighbors, where $s_i$ is the deficiency allowed for color $i$; thus if a vertex $v$ is nicely colored by a color $i$ which allows deficiency $s_i>0$, then an uncolored neighbor of $v$ can be colored by $i$.

\section{Special configurations}\label{reduce}
Let $(G, C_0)$ be a minimum counterexample to Theorem~\ref{main} with minimum $\sigma(G)=|V(G)|+|E(G)|$, where $C_0$ is a triangle or a $7$-cycle in $G$ that is precolored. For simplicity, let $F_k=\{f: \text{ $f$ is a $k$-face and } b(f)\cap C_0=\emptyset\}$,  $F_k'=\{f:  \text{ $f$ is a $k$-face and } |b(f)\cap C_0|=1\}$, and $F_k''=\{f:  \text{ $f$ is a $k$-face and }  |b(f)\cap C_0|=2\}$.

The following lemmas are shown in \cite{LLY14}.

\begin{prop}[Prop 3.1 in \cite{LLY14}]\label{prop3.1}
(a) Every vertex not on $C_0$ has degree at least $3$.\\
(b) A $k$-vertex in $G$ can have at most one incident $3$-face.\\
(c) No $3$-face and $4$-face in $G$ can have a common edge.
\end{prop}

\begin{lem}[Lemma 3.2 in \cite{LLY14}]\label{lem3.2}
The graph $G$ contains neither separating triangles nor separating $7$-cycles.
\end{lem}

\begin{lem}[Lemma 3.3 in \cite{LLY14}]\label{lem3.3}
If $G$ has a separating $4$-cycles $C_1=v_1v_2v_3v_4v_1$, then $ext(C_1)=\{b,c\}$ such that $v_1bc$ is a $3$-cycle. Furthermore, the $4$-cycle is the unique separating $4$-cycle.
\end{lem}

\begin{lem}[Lemma 3.4 in \cite{LLY14}]\label{lem3.4}
 If $x,y\in C_0$ with $xy\not\in E(C_0)$, then $xy\not\in E(G)$ and $N(x)\cap N(y)\subseteq C_0$.
\end{lem}

%%\begin{lem}[Lemma 3.5 in \cite{LLY14}]\label{lem3.5}
% Suppose that $f$ is a $4$-face with $b(f)=v_1v_2v_3v_4v_1$ and $v_1\in C_0$. Then, $v_3\not\in C_0$. Moreover, $|N(v_3)\cap C_0|=1$ if $f\in F_4''$, and $|N(v_3)\cap C_0|=0$ if $f\in F_4'$.
%\end{lem}

\begin{lem}[Lemma 3.6 in \cite{LLY14}]\label{lem3.6}
Let $u, w$ be a pair of diagonal vertices on a $4$-face. If at most one of $u$ and $w$ is incident to a triangle, $G[\{u,w\}]\in \mathcal{G}$.
\end{lem}

\begin{lem}[Lemma 3.7 in \cite{LLY14}]\label{4-face-property}
Let $f$ be a face in $F_4\cup F_4'$. Then
\begin{enumerate}
\item if $b(f)\cap C_0=\{u\}$, then each of $u$ and $w$ is incident to  a triangle.

\item  if $f=uvwx\in F_4$ is a face with $d(u)=d(w)=3$, then each of $v$ and $x$ is incident to a triangle.
\end{enumerate}
\end{lem}

By Lemma \ref{lem3.2}, we may assume that $C_0$ is the boundary of the outer face of $G$.

\begin{lem}\label{no-333-path}
In int$(C_0)$, let $v$ and $u$ be two adjacent $3$-vertices. Then each vertex in $(N(u)\cup N(v))\setminus \{u, v\}$  has degree at least $4$.
\end{lem}
\begin{proof}
Suppose to the contrary that $v_1$ is a neighbor of  $v$ that has degree 3. Let $G'=G-\{u,v\}$. By the minimality of $G$, $(G',C_0)$ is superextendable.   Recolor $v_1$ properly, and then color $u$ properly.  Now $v$ can be colored, or the three neighbors of $v$ are colored differently. In the latter case, $1$ or $2$ (say $1$) is used on $u$ or $v_1$. Then we color $v$ with $1$,  a contradiction.
\end{proof}

\begin{lem}\label{3-face-property}
Let $f=uvw$ be a face in $F_3$. Then  each of the following holds.
\begin{enumerate}
\item If $d(u)=d(v)=3$, then $d(w)\ge5$.

\item If $f$ is a $(3,3,5^+)$-face, then each pendant neighbor of $u$ or $v$ is  either on $C_0$ or has degree at least $4$.

\item If $f$ is a $(3,4,4)$-face, then the pendant neighbor of $u$ is either  on $C_0$ or has degree at least $4$  and at least one of the neighbors (not on $f$) of each $4$-vertex is either  on $C_0$ or has degree at least $4$. Consequently, a $4$-vertex cannot be incident to a $(3,4,4)$-face and a $(3,4,3,4^+)$-face from $F_4$.
\end{enumerate}
\end{lem}

\begin{proof}
(1) Suppose otherwise that $f=uvw$ is a $(3,3,4^-)$-face. Let $G'=G-\{u,v\}$. It follows that $\sigma(G')<\sigma(G)$. By the minimality of $G$, $(G', C_0)$ is superextendable. Recolor $w$ properly and then color $u$ properly.  Then $v$ can be colored, or $N(v)$ contains three different colors. In the latter case, $1$ or $2$ (say $1$) is used on $u$ or $w$, then we can color $v$ with $1$, a contradiction.

(2) Let $f=uvw$ be a $(3,3,5^+)$-face. Let $u'$ be the pendant neighbor of $u$. Assume that $u'$ is not on $C_0$. Suppose otherwise that $d(u')=3$. By Lemma \ref{no-333-path}, each vertex in $(N(u)\cup N(v))\backslash \{u,v\}$ has degree at least $4$. So $d(u')\ge4$, a contradiction.

(3) Let $f=uvw$ be a $(3,4,4)$-face. Let $u'$ be the pendant neighbor of $u$. Assume that $u'$ is not on $C_0$. Suppose otherwise that $d(u')=3$. By the minimality of $G$, $(G-\{u,u'\}, C_0)$ is superextendable. Color $u'$ properly.  Then $u$ can be colored, or $N(u)$ contains three different colors. In the latter case, if $u'$ is colored with $1$ or $2$, then we color $u$ with the color of $u'$. Thus, we may assume that $u'$ is colored  with $3$, and assume that $v,w$ are colored with $1,2$ respectively and neither is nicely colored.  If the neighbors of $v$ not on $f$ are colored with $1$ and $2$, then we recolor $v$ with $3$ and color $u$ with $1$. So, we may assume that they are colored with $1$ and $3$. Similarly, we may assume that the neighbors of $w$ not on $f$ are colored with $2$ and $3$. Now we switch the color of $v$ and $w$, and color $u$ with $1$, a contradiction.

Now let $v_1,v_2$ be the two neighbors of $v$ not on $f$. Suppose otherwise that $d(v_1)=d(v_2)=3$ and $v_1,v_2\not\in V(C_0)$.  By the minimality of $G$, $(G-\{u,v,w,v_1,v_2\}, C_0)$ is superextendable.  We properly color $v_1,v_2, w$ and $u$ in order. Then $v$ can be properly colored, or $N(v)$ has three different colors. In the latter case, only one vertex in $\{u, w, v_1,v_2\}$ is colored with $1$ or $2$ (say $1$), so we color $v$ with $1$, a contradiction.
\end{proof}

\begin{lem}\label{3-face-and-4-face}
Let $v$ be a $k$-vertex with $N(v)=\{v_i:i\in[k]\}$. Then each of the following holds.
\begin{enumerate}
\item For $k=4$, if $v$ is incident to a $(3,4,4)$-face $f_1=v_1vv_2$ from $F_3$ and a $4$-face $f_2=vv_3uv_4$ from $F_4$ with $d(u)=3$, then both $v_3$ and $v_4$ are incident to triangles. Consequently, $f_2$ cannot be a $(3,3,4,4^+)$-face.

\item For $k=5$, let $v$ be incident to a $(3,4^-,5)$-face $f_1=v_1v_2v$ from $F_3$ and two $4$-faces $f_2=vv_3uv_4$ and $f_3=vv_4wv_5$ from $F_4$. If $d(u)=d(w)=3$, then at least two vertices in $\{v_3,v_4,v_5\}$ are incident to triangles.
\end{enumerate}
\end{lem}

\begin{proof}
(1) Suppose otherwise that at most one vertex in $\{v_3, v_4\}$ is incident to a triangle. Let $G'=G[\{v_3, v_4\}]$ and $v'$ be the new vertex. By Lemma~\ref{lem3.6}, $G'\in \mathcal{G}$. Then $(G'-\{v, v_1, v_2, u\},C_0)$ is superextendable. We color $v_3$ and $v_4$ with the color of $v'$, then properly color $v_2, v_1,u$ in order. Then $v$ can be properly colored, or $N(v)$ has three different colors. In the latter case, $1$ or $2$ (say $1$) is used on $v_1$ or $v_2$, so we color $v$ with $1$, a contradiction.

(2) Suppose otherwise that at most one vertex in $\{v_3,v_4,v_5\}$ is incident to a triangle. Let $G'=G[\{v_3,v_4,v_5\}]$, and let $v'$ be the new vertex. By lemma \ref{lem3.6}, $G'\in\mathcal{G}$. Then $(G'-\{v,u,w\},C_0)$ is superextendable.  Color $v_3,v_4,v_5$ with the color on $v'$,  and then properly color $u$ and $w$ since $d(u)=d(w)=3$. We uncolor $v_1,v_2$ and then recolor $v_2,v_1$ properly in the order.  Then $v$ can be properly colored, or $N(v)$ has three different colors.  In the latter case, $1$ or $2$ (say $1$) is used on $v_1$ or $v_2$, so we can color $v$ with $1$, a contradiction.
\end{proof}

We first prove the following useful lemma.

\begin{lem}\label{claimA}
Let $v$ be a $4$-vertex in $int(C_0)$ with $N(v)=\{v_i:i\in[4]\}$. If $v$ is incident to two $4$-faces that share an edge, then there is no $t$-path from $v_i$ to $v_{i+2}$ with $t\in \{1, 2, 3, 5\}$, where the subscripts of $v$ are taken modulo $4$.
\end{lem}

\begin{proof}
As $v$ is incident to two $4$-faces that share an edge, in any embedding, $v_i$ and $v_{i+2}$ cannot be in the same face, for otherwise, they will be in a separating $4$-cycle, contrary  to Lemma \ref{lem3.3}.  Suppose otherwise that $P$ is a $t$-path from $v_i$ to $v_{i+2}$ with $t\in \{1,2,3,5\}$. Consider cycle $C=v_iPv_{i+2}vv_i$.  If $t=1$ or $5$, then $C$ is a $3$- or $7$-cycle separating $v_{i+1}$ and $v_{i+3}$, a contradiction to Lemma~\ref{lem3.2};  if $t=2$, then $C$ is a $4$-cycle separating $v_{i+1}$ and $v_{i+3}$, a contradiction to Lemma~\ref{lem3.3};  if $t=3$, then $C$ is a $5$-cycle, a contradiction to $G\in \mathcal{G}$.
\end{proof}

Let $v$ be a $4$-vertex with  its neighbor $v_1, v_2, v_3, v_4$ in the clockwise order in the embedding.   Then $v$ is called {\em $(v_i, v_{i+2})$-behaved} if at most one of $v_i$ and $v_{i+2}$ is incident to a triangle.

\begin{lem}\label{lem3.6b}
Let $v$ be a $4$-vertex in $int(C_0)$ with $N(v)=\{v_i:i\in[4]\}$. Then each of the following holds.
\begin{enumerate}
\item If $v$ is incident to two $4$-faces $f_i=vv_iu_iv_{i+1}$ and $f_{i+1}=vv_{i+1}u_{i+1}v_{i+2}$ with $f_i, f_{i+1}\in F_4$, and at most one of $\{v_i, v_{i+1}, v_{i+2}\}$ is incident to a triangle, then $d(u_i)\ge 4$ or $d(u_{i+1})\ge 4$, where the subscripts of $u$ and $v$ are taken modulo 4.

\item If $v$ is incident to two $4$-faces $f_i=vv_iu_iv_{i+1}$ and $f_{i+2}=vv_{i+2}u_{i+2}v_{i+3}$ with $f_i, f_{i+2}\in F_4$, and at most one vertex from each of $\{v_i,v_{i+1}\}$ and $\{v_{i+2},v_{i+3}\}$ is incident to a triangle, then $d(u_i)\ge 4$ or $d(u_{i+2})\ge 4$, where the subscripts of $u$ and $v$ are taken modulo 4.

\item The vertex $v$ is incident to at most one $(3,3,4,4^+)$-face from $F_4$.

\item Let $v$ be incident to two $4$-faces that share an edge. If $v$ is $(v_1,v_3)$-behaved and $(v_2,v_4)$-behaved, then none of the $4$-faces can be $(3,3,4,4^+)$-face.
\end{enumerate}
\end{lem}

\begin{proof}
(1) By symmetry we assume that $i=1$. Suppose otherwise that $d(u_1)=d(u_2)=3$.  Let $G'=G[\{v_1,v_2,v_3\}]$.  Since at most one vertex in $\{v_1,v_2,v_3\}$ is incident to a triangle, by Lemma \ref{lem3.6}, $G'\in\mathcal{G}$. Thus, $(G',C_0)$ is superextendable. Color $v_1$,$v_2$ and $v_3$ with the color of the resulting vertex of identification and then we can recolor $u_1$, $u_2$ and $v$ properly, a contradiction.

(2) By symmetry we assume that $i=1$. Suppose otherwise that $d(u_1)=d(u_3)=3$.   Let $G'=G[\{v_1,v_2\},\{v_3,v_4\}]$. Let $v'$ and $v''$ be the new vertices by identifying $v_1$ with $v_2$, and $v_3$ with $v_4$, respectively.  Since at most one vertex from each of $\{v_1,v_2\}$ and $\{v_3,v_4\}$ is incident to a triangle,  by Lemma~\ref{lem3.6},   $G'\in\mathcal{G}$. Thus $(G',C_0)$ is superextendable. Color $v_1,v_2$ with the color of $v'$ and color $v_3,v_4$ with the color of $v''$, then we can recolor $v$, $u_1$ and $u_3$ properly, a contradiction.

(3)  Suppose otherwise that $v$ is incident to at least two $(3,3,4,4^+)$-faces $f_1, f_2\in F_4$.  If $f_1 $ and $f_2$ share an edge, let $f_1=vv_1u_1v_2$ and $f_2=vv_2u_2v_3$, then $d(u_1)=d(u_2)=3$. We first show that $d(v_2)\ge4$. Assume that $d(v_2)=3$. Since $u_1$ and $v_2$ are two adjacent $3$-vertices in $int(C_0)$, so by Lemma \ref{no-333-path}, $(N(u_1)\cup N(v_2))\backslash \{u_1,v_2\}$ has degree at least $4$, which implies that $d(u_2)\ge4$, a contradiction. Thus $f_1$ is a $(3,3,4,4^+)$-face with $d(u_1)=d(v_1)=3$ and $f_2$ is a $(3,3,4,4^+)$-face with $d(u_2)=d(v_3)=3$. By Propositin \ref{prop3.1}(c), none of $v_1$ and $v_3$ is incident to a triangle. So by (1), $d(u_1)\ge4$ or $d(u_2)\ge4$, a contradiction to $d(u_1)=d(u_2)=3$. If $f_1$ and $f_2$ do not share an edge, then it contradicts to (2).

(4) Assume that $v$ is incident to a $(3,3,4,4^+)$-face $f_1$. Then by symmetry  $d(v_1)=d(u_1)=3$ or $d(u_1)=d(v_2)=3$. First we assume that $d(v_1)=d(u_1)=3$. Let $G'=G-v$ and $H=G'[\{v_2, v_4\}]$.   By Lemma~\ref{claimA}, there is no $t$-path from $v_2$ to $v_4$ with $t\in \{1,2,3,5\}$, so $H$ contains no $5$-cycle and no new triangles, in addition to the fact that $G$ is $(v_2,v_4)$-behaved, $H$ has no intersecting triangles, therefore $H\in \mathcal{G}$. Thus $(H,C_0)$ is superextendable. Color $v_2$ and $v_4$ with the color of the new vertex, then $v$ can be colored properly, or $N(v)$ has three different colors.  Consider the latter case.   Recolor $u_1, v_1$ properly in the order. If $v_1$ is colored with $1$ or $2$, then we color $v$ with the color of $v_1$; if $v_1$ is colored $3$, then color $v$ with $3$ and recolor $v_1$ with the color of $u_1$.  In either case, we reach a contradiction. Similar to the above argument, $v$ cannot be incident to a $(3,3,4,4^+)$-face with $d(u_1)=d(v_2)=3$.
\end{proof}

For $k=4,5$, we call a $k$-vertex  in $int(C_0)$ to be {\em poor} if it is incident to $k$ $4$-faces from $F_4$.  If a $k$-vertex is not poor, then we call it {\em rich}.

\begin{lem}\label{lem3.13}
Let $v$ be a poor $4$-vertex with $N(v)=\{v_i: i\in[4]\}$ and four incident $4$-faces $f_i=vv_iu_iv_{i+1}$ for $i\in[4]$, where the subscripts of $v$ and $u$ are taken modulo $5$.  Furthermore,  $v$ is $(v_1,v_3)$-behaved.  If either $d(v_2)=3$ or $d(v_2)=4$ and $v_2$ is $(u_1,u_2)$-behaved, then $d(v_4)\ge 5$, or $d(v_4)=4$ and $v_4$ is not $(u_3,u_4)$-behaved.
\end{lem}

\begin{proof}
Suppose to the contrary that $d(v_4)=3$ or $d(v_4)=4$ and $v_4$ is $(u_3,u_4)$-behaved.

Consider that $d(v_2)=d(v_4)=3$.  Let $G'=G-v$ and $H=G'[\{v_1, v_3\}]$.   By Lemma~\ref{claimA}, there is no $t$-path from $v_1$ to $v_3$ with $t\in \{1,2,3,5\}$. It follows that $H$ contains no $5$-cycle and no new triangles. In addition to the fact that $v$ is $(v_1,v_3)$-behaved, $H$ has no intersecting triangles. Therefore, $H\in \mathcal{G}$.  Thus $(H,C_0)$ is superextendable. Color $v_1$ and $v_3$ with the color of the new vertex, and recolor $v_2,v_4$ properly.  Then $1$ or $2$ (say $1$) is used on $v_2$ or $v_4$.  Now color $v$ with $1$, a contradiction.

By symmetry, consider that $d(v_2)=3$ and $d(v_4)=4$.  Let $G'=G-\{v, v_4\}$ and let $H=G'[\{v_1, v_3\}, \{u_3, u_4\}]$. Let $v'$ and $v_4'$ be the new vertices by identifying $v_1$ with $v_3$ and $u_3$ with $u_4$, respectively.  As above, there is no $5$-cycle or new $3$-cycle containing $v'$ or $v_4'$.  Furthermore, if there is a $3$-cycle, $5$-cycle containing $v'$ and $v_4'$, then there is a $2$-path or a $4$-path from $\{v_1,v_3\}$ to $\{u_3,u_4\}$, thus there is $5$-cycle or separating $7$-cycle in $G$, a contradiction.  Therefore, $H\in\mathcal{G}$.   Note that now $(H,C_0)$ is superextendable.  Color $v_1,v_3$ with the color of $v'$ and color $u_3,u_4$ with the color of $v_4'$, then properly color $v_2,v_4$.  Now $v$ can be colored, or $N(v)$ contains three different colors. In the latter case, $1$ or $2$ (say $1$) is used on $v_2$ or $v_4$, then color $v$ with $1$, a contradiction.

Consider $d(v_2)=d(v_4)=4$. Let $G'=G-\{v, v_2, v_4\}$, and let $H=G'[\{u_1,u_2\},\{v_1,v_3\},\{u_3,u_4\}]$.  Let $v_2',v',v_4'$ be the new vertices by identifying $u_1$ with $u_2$, $v_1$ with $v_3$ and $u_3$ with $u_4$, respectively. As shown above, there is no $3$-cycle or $5$-cycle containing one of $v_2', v', v_4'$, or the pairs in $\{v_2',v'\}, \{v',v_4'\}, \{v_2',v_4'\}$.  If there is a $3$-cycle or $5$-cycle containing $v_2',v'$ and $v_4'$ then there is $1$- or $3$-path from $v_2'$ to $v_4'$ or a $2$-path from $v'$ to $v_4'$, but in either case, there is a $5$-cycle or a separating $7$-cycle, a contradiction. Thus, $(H,C_0)$ is superextendable. Color the vertices with the color of their resulting vertex, respectively, then color $v_2,v_4$ properly.  Now $v$ can be colored, or $N(v)$ contains three different colors. In the latter case, $1$ or $2$ (say $1$) is used on $v_2$ or $v_4$, then we color $v$ with $1$, a contradiction.
\end{proof}

\begin{lem}\label{lem3.14}
Let $v$ be a poor $5$-vertex with  $N(v)=\{v_i:  i\in [5]\}$ and five incident $4$-faces $f_i=vv_iu_iv_{i+1}$ for $i\in [5]$, where the subscripts of $u$ and $v$ are taken modulo 5. Suppose that at most one vertex in $N(v)$ is incident with a triangle.  Then each of the following holds.
\begin{enumerate}
\item If $d(u_i)=d(v_i)=3$ for some $i\in [5]$, then $d(u_j)\ge 4$ for $j\in [5]-\{i\}$.

\item At most two vertices in $\{u_i: i\in [5]\}$ have degree $3$.

\item Let $d(u_i)=3$. If $v_j$ has degree 3 or is a 4-vertex with $(u_{j-1}, u_j)$-behaved, then $d(v_k)\ge5$, or $d(v_k)=4$ and $v_k$ is not $(u_{k-1},u_k)$-behaved, where $\{j, k\}=\{i-1, i+2\}$.
\end{enumerate}
\end{lem}

\begin{proof}
(1)Without loss of generality,  We may assume that $i=1$.  By Lemma~\ref{no-333-path}, $d(u_5)\ge 4$.  Suppose otherwise that $d(u_j)=3$ for some $j\not=1,5$.   Let $H=G'[\{v_j, v_{j+1}, v_{j+3}\}]$, where $G'=G-v$. By Lemma~\ref{lem3.6} and ~\ref{lem3.13},  $H\in \mathcal{G}$.  So $(H,C_0)$ is superextendable.  In $G'$, color $v_j, v_{j+1}, v_{j+3}$ with the color of the resulting vertex, and uncolor $u_j, u_1, v_1$ and recolor them properly in the order, we get a desired coloring of $G'$.  Now $v$ can be properly colored, or $N(v)$ contains three different colors.  In the latter case, if $v_1$ is colored with $1$ or $2$, then color $v$ with the color of $v_1$; if $v_1$ is colored with $3$, then color $v$ with $3$ and recolor $v_1$ with the color of $u_1$, a contradiction.

(2) Suppose otherwise that at least three vertices in $\{u_i: i\in[5]\}$ have degree $3$.  By symmetry,  $u_i, u_{i+1}, u_{i+2}$ have degree $3$ or $u_i, u_{i+1}, u_{i+3}$ have degree $3$ for some $i\in [5]$. We may assume that $i=1$. Let $d(u_1)=d(u_2)=d(u_3)=3$. Consider $H=G[\{v_1, v_2, v_3, v_4\}]$.  By Lemma~~\ref{lem3.6},  $H\in \mathcal{G}$.  So $(H,C_0)$ is superextendable.  In $G$, color $v_1, v_2, v_3, v_4$ with the color of the resulting vertex and recolor $u_1, u_2, u_3$ properly and finally color $v$ properly, a contradiction. Let $d(u_1)=d(u_2)=d(u_4)=3$. Consider $H=G[\{v_1, v_2, v_3\}, \{v_4, v_5\}]$. Let $v'$ and $v''$ be the resulting vertices by identifying $v_1, v_2, v_3$  and $v_4, v_5$, respectively. By Lemma~\ref{lem3.6}, $H\in \mathcal{G}$.  So $(H,C_0)$ is superextendable.  In $G$, color $v_1, v_2, v_3$ with the color of $v'$ and color $v_4, v_5$ with the color of $v''$ and recolor $u_1, u_2, u_4$ properly, and now $v$ can be properly colored, a contradiction.

(3)Without loss of generality,  We assume that $i=2$ and $d(u_2)=3$. Let $H=G[\{v_2,v_3\}]-u_2$ and the resulting vertex be $v'$. By symmetry, let $j=i-1=1$ and $k=i+2=4$. Suppose to the contrary that $d(v_4)=3$ or $d(v_4)=4$ and $v_4$ is $(u_3, u_4)$-behaved. By the proof of Lemma \ref{lem3.13}, we can get a desired coloring of $H$ and the color of $v$ is different from the color of $v'$. Then we color $v_2$ and $v_3$ with the color of $v'$ and color $u_2$ properly, a contradiction.
\end{proof}

\section{Discharging Procedure}\label{discharging}
In this section, we will finish the proof of the main theorem by a discharging argument.  Let the initial charge of vertex $v\in G$ be $\mu(v)=2d(v)-6$, and the initial charge of face $f\not=C_0$ be $\mu(f)=d(f)-6$ and $\mu(C_0)=d(C_0)+6$. Then
$$ \sum_{v\in V(G)} \mu(v)+ \sum_{f\in F(G)} \mu(f)=0.$$

We will use the following special $4$-faces from $F_4$ in the discharging.

\begin{itemize}

\item A $(3,4,4,5)$-face is {\em special} if none of the $4$-vertices is incident to triangles.

\item A $(3,4,4,5)$-face is {\em weak} if exactly one of the $4$-vertices is incident to a triangle.

\item A $(3,4,5,5)$- or $(3,5,4,5)$- or $(3,5,5,5)$-face is {\em special} if the $5$-vertices on the face are poor.

\item A $(4,4,4,5)$-face is {\em special} if the $5$-vertex and the $4$-vertices adjacent to the $5$-vertex are poor.

\item A $(4,4,5,5)$-face is {\em special} if the $4$-vertices and $5$-vertices are poor.

\item A $4$-face is {\em rich} if it contains two rich $5$-vertices or $6^+$-vertices.
\end{itemize}

The discharging rules are as follows.

\begin{enumerate}[(R1)]
\item Let $v\not\in C_0$. Then $v$ gives charges in the following ways:
\begin{enumerate}

\item[(R1.1)]  $d(v)=4$
\begin{enumerate}
\item[(R1.1.1)] If $v$ is rich, then $v$ gives $\frac{5}{4}$ to each incident $(3,4,4)$-face from $F_3$ and $1$ to other $3$-faces from $F_3$, $\frac{1}{2}$ to each pendant $3$-face from $F_3$,  $1$ to each incident $(3,3,4,4^+)$-face from $F_4$. Furthermore, if $v$ is incident to a triangle, then $v$ gives $\frac{3}{4}$ to its incident $4$-face (other than $(3,3,4,4^+)$-face from $F_4$); if $v$ is not incident to a triangle, then $v$ distributes the remaining charges only to other incident $4$-faces form $F_4$ evenly. 

\item[(R1.1.2)] If $v$ is poor, then $v$ gives $max\{0, \frac{2-w(f)}{|Q|}\}$ to $f$, where $v$ is on $4$-face $f$ and $Q$ is the set of  poor $4$-vertices on $f$, and $w(f)$ is the weight that $f$ receives from vertices not in $Q$.
\end{enumerate}

\item[(R1.2)] $d(v)=5$

\begin{enumerate}
\item[(R1.2.1)] If $v$ is rich, then $v$ gives $2$ to each incident $(3,4^-,5)$-face from $F_3$, and $\frac{3}{2}$ to other incident $3$-faces from $F_3$, $\frac{1}{2}$ to each pendant $3$-face from $F_3$.  Furthermore, if $v$ is incident to a triangle, then $v$ gives $1$ to its incident $4$-face; if $v$ is not incident to a triangle, then $v$ distributes the remaining charges only to other incident $4$-faces form $F_4$ evenly. 

\item[(R1.2.2)] If $v$ is poor, then $v$ gives $1$ to each incident $(3,3,5,4^+)$-face or $(3,4,5,4)$-face or special $(3,4,4,5)$-face, $\frac{3}{4}$ to each incident special $(3,4,5,5)$-, $(3,5,4,5)$-, $(3,5,5,5)$-, $(4,4,5,5)$-, $(4,4,4,5)$-face,  or weak $(3,4,4,5)$-face, $0$ to a rich $4$-face, and $\frac{1}{2}$ to each other incident $4$-face.
\end{enumerate}

\item[(R1.3)] Each $6^+$-vertex gives $2$ to each incident $3$-face, $\frac{1}{2}$ to each pendant $3$-face, and distributes the remaining charges to incident $4$-faces evenly.

\end{enumerate}

\item Each $v\in C_0$ gives $\frac{1}{2}$ to each pendant face from $F_3$,  $1$ to each incident face from $F_4''$ , $\frac{3}{2}$ to each incident face from $F_3''$ or $F_4'$, and $3$ to each incident face from $F_3'$.

\item $C_0$ gives $2$ to each $2$-vertex on $C_0$, $\frac{3}{2}$ to each $3$-vertex on $C_0$, and $1$ to each $4$-vertex on $C_0$.  In addition, if $C_0$ is a $7$-face with six $2$-vertices, then it gains $1$ from the incident face.
\end{enumerate}

The following useful facts are from the rules.

\begin{lem}\label{claim4.1}
The vertices and faces mentioned in this lemma are  disjoint from $C_0$.
\begin{enumerate}[(1)]
\item If a $4$-vertex is incident to a triangle, then it gives $1$ to each incident $(3,3,4,4^+)$- or $(3,4,3,4^+)$-face, and at least $\frac{3}{4}$ to each other $4$-face.

\item Each rich $4$-vertex gives at least $\frac{1}{2}$ to each incident $4$-face, and if it is not incident to $(3,3,4,4^+)$-face, then it gives at least $\frac{2}{3}$ to each incident $4$-face.

\item Let $f=uvwx$ be a $(3,4,4,4)$-face with $w$ not incident with a triangle. Then each rich $4$-vertex on $b(f)$ gives at least $\frac{2}{3}$ to $f$.

\item A rich $5$-vertex gives at least $1$ to each incident $4$-face. Moreover, if such a 5-vertex is incident to a triangle that is not a $(3,4^-,5)$-face, then  it gives at least $\frac{5}{4}$ two each incident $4$-face.  A $6^+$-vertex gives at least $1$ to each incident $4$-face. Moreover, if such a $6^+$-vertex is incident to a triangle, then it gives at least $\frac{4}{3}$ to each incident $4$-face.

\item Let $v$ be a poor $4$-vertex on a $4$-face $f$.  Then $v$ gives at most $1$ to each incident $(3,3,4,4^+)$-face, at most $\frac{2}{3}$ to each incident $(3,4,4,4)$-face, at most $\frac{1}{4}$ to a $(4,4^+, 4^+, 4^+)$-face that is adjacent to a triangle and at most $\frac{1}{2}$ to each other incident $4$-face.
\end{enumerate}
\end{lem}

\begin{proof}
(1) By (R1.1.1), we just need to show that when $v$ is incident to a $(3,4,3,4^+)$ or a $(3,3,4,4^+)$-face, $v$ cannot be incident to a $(3,4,4)$-face. But this is true by Lemma~\ref{3-face-property}(3) and Lemma~\ref{3-face-and-4-face} (1).

(2) Let $v$ be a rich $4$-vertex, note that $v$ is incident to at most three $4$-faces. Suppose that $v$ is incident to exactly one $4$-face $f$. So if $v$ is incident to a triangle, then by (R1.1.1), it gives at least $2-\frac{5}{4}=\frac{3}{4}$ to $f$; if $v$ is not incident to a triangle but adjacent to pendant triangles, then it gives at least $2-2\cdot \frac{1}{2}=1$ to $f$; otherwise, $v$ gives at least $2$ to $f$.

Let $v$ be incident to exactly two $4$-faces. Since $G$ has no 5-cycle, $v$ is not incident to a triangle. If $v$ is not adjacent to a pendant triangle, then it gives at least $1$ to each incident $4$-face. Let $v$ be adjacent to a pendant triangle.  If $v$ is not incident to $(3,3,4,4^+)$-face, then by (R1.1.1), $v$ gives $\frac{2-\frac{1}{2}}{2}=\frac{3}{4}$ to each $4$-face;  if $v$ is incident to a $(3,3,4,4^+)$, then by Lemma~\ref{lem3.6b}(3), it is incident to exactly one $(3,3,4,4^+)$-face. By (R1.1.1), $v$ gives $2-\frac{1}{2}-1=\frac{1}{2}$ to the other $4$-face.

If $v$ is incident to exactly three $4$-faces, by Lemma~\ref{lem3.6b}(3), it is incident to at most one $(3,3,4,4^+)$-face. If $v$ is incident to a $(3,3,4,4^+)$-face, then by (R1.1.1), it gives at least $\frac{2-1}{2}=\frac{1}{2}$ to each incident $4$-face, otherwise, $v$ gives at least $\frac{2}{3}$ to each incident 4-face.

(3) By symmetry suppose that $v$ or $w$ is rich $4$-vertices. By Lemma \ref{lem3.6b}(1) and (4) $v$ or $w$ cannot be incident to a $(3,3,4,4^+)$-face that share an edge with $f$ since $w$ is not incident to a triangle. By (R1.1.1) $v$ or $w$ gives at least $\frac{2}{3}$  to $f$.

(4) Let $v$ be a rich $5$-vertex that is incident to $t_3\le 1$ triangles and $s$ pendant $3$-faces. Then $v$ is incident to at most ($5-2t_3-s-1$) $4$-faces. By (R1.2.1), $v$ gives at least $\frac{4-2t_3-\frac{1}{2}s}{5-2t_3-s-1}\ge 1$ to each incident $4$-face.  In particular, if $v$ is incident to a triangle that is not a $(3,4^-,5)$-face, then by (R1.2.1), $v$ gives at least $\frac{4-\frac{3}{2}-\frac{1}{2}s}{5-2-s-1}\ge \frac{5}{4}$ to each incident $4$-face.

Similarly, if $v$ is a $t$-vertex with $t\ge 6$ that is incident to $t_3\le 1$ triangles and $s$ pendant $3$-faces, then $v$ is incident to at most ($t-2t_3-s$) $4$-faces. By (R1.3),  $v$ gives at least $\frac{2t-6-2t_3-\frac{1}{2}s}{t-2t_3-s}=\frac{t-2t_3-\frac{1}{2}s+(t-6)}{t-2t_3-s}\ge 1$. Moreover, if $t_3=1$, then $v$ is incident to at most ($t-s-3$) $4$-faces.  In this case, $v$ gives at least $\frac{(2t-6)-2-\frac{1}{2}s}{t-s-3}\ge \frac{2t-8}{t-3}\ge \frac{4}{3}$ to each incident $4$-face.

(5) First assume that $f$ is a $(3,3,4,4^+)$-face with $d(x)\ge4$. If $x$ is also a poor $4$-vertex, then by (R1.1.2) both $x$ and $v$ give $1$ to $f$. If $x$ is not a poor $4$-vertex, then by (R1.1.1),(R1.2.2) and (4), $x$ gives at least $1$ to $f$. In either case, by (R1.1.2) $v$ gives at most $1$ to $f$.

 Second, assume that $f$ is a $(3,4,4,4)$-face. Since $v$ is poor, the $4$-vertex not adjacent to $3$-vertex on $f$ is not incident to a triangle. By (3) and (R1.1.2), $v$ gives at most $\frac{2}{3}$ to $f$.

 Next, assume that  $f=vuwx$ is a $(4,4^+,4^+,4^+)$-face that is adjacent to a triangle.   Let $w$ be incident to a triangle.  If $d(w)=4$, then both $u$ and $x$ are rich.  By (2), (4) and (R1.2.2), $u$ and by (1) $x$ each gives at least $\frac{1}{2}$ to $f$ and $w$ gives at least $\frac{3}{4}$ to $f$. So by (R1.1.2) $v$ gives at most $\frac{1}{4}$ to $f$.  If $d(w)=5$, then $u$ or $x$ is not poor. We assume, without loss of generality, that $x$ is not poor. By (2)(4) and (R1.2.2),  $x$   gives at least $\frac{1}{2}$ to $f$.  In this case,   $u$ and $v$ may be both poor. It follows  by (4) and (R1.1.2) that $v$ gives at most $\frac{2-1-\frac{1}{2}}{2}=\frac{1}{4}$ to $f$. If $d(w)\ge 6$, then by (4) $w$ gives at least $\frac{4}{3}$ to $f$. In this case, each of $u,x$ and $v$ may be poor. By (R1.1.2), $v$ gives at most $\frac{2-\frac{4}{3}}{3}=\frac{2}{9}<\frac{1}{4}$ to $f$.  Now by symmetry let $u$ be incident to a triangle.  Then $d(u)\ge 5$. It follows that  either $d(u)\ge 6$ or $w$ is not poor. In the former case, similarly, we can show that $v$ gives at most $\frac{2}{9}<\frac{1}{4}$. In the latter case, by (2)(4) and (R1.2.2) $w$ gives at least $\frac{1}{2}$ to $f$ and $u$ gives at least 1 to $f$. Note that $x$ may be poor. Thus by (R1.1.2), $v$ gives at most $\frac{2-1-\frac{1}{2}}{2}=\frac{1}{4}$ to $f$.

 Finally, assume that $f$ is a $4$-face which is neither  $(3,3,4,4^+)$ nor  $(3,4,4,4)$-face. By Lemma \ref{4-face-property}(2), the number of $3$-vertices on $f$ is at most two. Since $f$ is not $(3, 3, 4, 4^+)$-face, the number of 3-vertices on $f$ is at most one.  First consider that $f$ contains no $3$-vertex. If $f$ is a rich $4$-face, then by (4) each of the two rich $5$-vertices or $6^+$-vertices gives at least $1$ to $f$. In this case, by (R1.1.2)  $v$ gives $0$ to $f$. If $f$ is not a rich $4$-face, then by (2) (4) and (R1.2.2), each of $4^+$-vertices on $f$ not in $Q$ gives at least $\frac{1}{2}$ to $f$, where $Q$ is the set of poor $4$-vertices on $f$. By (R1.1.2), $v$ gives at most $\frac{2-\frac{1}{2}(4-|Q|)}{|Q|}=\frac{1}{2}$ to $f$.

Next consider that $f$ contains one $3$-vertex.  Since $f$ is not (3, 4, 4, 4),  it contains at least one  $5^+$-vertex. On the other hand, since $f$ contains one 3-vertex and one 4-vertex $v$, $f$ contains at most two $5^+$-vertices.  Assume first that $f$ contains exactly  two $5^+$-vertices. If both $5^+$-vertices are rich $5$-vertices or $6^+$-vertices, by (4), each of them gives 1 to $f$. By (R1.1.2), $v$ gives 0 to $f$. If exactly one of $5^+$-vertex is poor $5$-vertex. Then by (4) and (R1.2.2), the poor $5$-vertex gives at least $\frac{1}{2}$ to $f$ and the other $5^+$-vertex gives at least $1$ to $f$. Thus, by (R1.1.2) $v$ gives at most $\frac{1}{2}$ to $f$. Thus, we may assume that both of the $5^+$-vertices must be poor $5$-vertices.  It follows that $f$ is a special $(3,4,5,5)$ or $(3,5,4,5)$-face. By (R1.1.2) and (R1.2.2), $v$ gives at most $2-2\cdot \frac{3}{4}=\frac{1}{2}$ to $f$.

  Thus, assume that $f$ contains one $5^+$-vertex. It follows  that $f$ is a $(3,4,4,5^+)$ or $(3,4,5^+,4)$-face. If the $5^+$-vertex is not poor 5-vertex, then by (4), it gives at least $1$ to $f$. If the other 4-vertex is rich, then by (2), it gives $\frac{1}{2}$ to $f$. Thus,  by (R1.1.2), $v$ gives at most $\frac{1}{2}$ to $f$. If the other 4-vertex is poor, then by (R1.1.2) again, $v$  gives at most $\frac{1}{2}$ to $f$.  Thus, we may assume that the $5^+$-vertex is a poor 5-vertex. In this case,  $f$ is  a special $(3,4,4,5)$-face or  weak $(3, 4,4,5)$-face or $(3, 4, 5, 4)$-face.    By (R1.2.2), (R1.1.2), (1) and (2), $v$ gives at most $max\{\frac{2-1}{2}, 2-2\cdot\frac{3}{4}\}=\frac{1}{2}$ to $f$.
\end{proof}

Now we shall show that each $x\in V(G)\cup F(G)$ other than $C_0$ has final charge $\mu^*(x)\ge 0$ and $\mu^*(C_0)>0$.

First we consider vertices in $int(C_0)$. Note that $int(C_0)$ contains no $2^-$-vertices by Proposition~\ref{prop3.1}. As $3$-vertices in $int(C_0)$ is not involved in the discharging process,  they have final charge $2\cdot 3-6=0$. By (R1.3),  $6^+$-vertices have nonnegative final charges.  Thus, we are left with $4$-vertices and $5$-vertices in $int(C_0)$.

In Lemmas~\ref{le42} -\ref{le43}, when we discuss the case that $v$ is a poor $k$-vertex for $k=4, 5$, we assume that $N(v)=\{v_i: i\in [k]\}$ and  $f_i=vv_iu_iv_{i+1}$ for $i\in [k]$ be the $k$ incident $4$-faces of $v$ (the subscripts of $u$ and $v$ are taken modulo $k$). We further assume that $v_1, v_2, \ldots, v_k$ are in the clockwise order in the embedding.

\begin{lem}
\label{le42}
Each $4$-vertex $v\in int(C_0)$ has nonnegative final charge.
\end{lem}

\begin{proof}
First suppose that $v$ is rich.  Note that when $v$ is incident with a $3$-face, it is incident with at most one $4$-face and at most one $3$-face, since $G$ has no $5$-cycle and intersecting $3$-cycle.  By Lemma~\ref{lem3.6b}(3), $v$ is incident to at most one $(3,3,4,4^+)$-face from $F_4$.  So by (R1.1.1), $v$ gives out more than $2$ only if $v$ is incident to a $(3,4,4)$-face from $F_3$ and a $(3,3,4,4^+)$-face from $F_4$, which is impossible by Lemma~\ref{3-face-and-4-face} (1),   or a $(3,4,4)$-face from $F_3$ and two pendant $3$-faces from $F_3$, which is also impossible by Lemma~\ref{3-face-property}.  So $v$ gives out at most $2$, and its final charge is at least $2\cdot 4-6-2=0$.

%Assume first that $v$ is incident with a 3-face. In this case, since $G$ is 5-cycle free, $v$ is incident with at most one 4-face. If this triangle is not $(3,4,4)$-face, then $v$ has at most two pendant $3$-faces or incidents to a $4$-face. So by (R1.1.1), $\mu^*\ge2-1-1=0$. If this triangle is $(3, 4, 4)$-face, then  by Lemma~\ref{3-face-and-4-face} (1), $v$ cannot be incident to both a $(3,4,4)$-face and a $(3,3,4,4^+)$-face. Thus, if $v$ is incident with a 4-face $f'$, then $f'$ is not a $(3, 3,4,4^+)$-face. By (R1.1.1), $v$ gives $\frac{3}{4}$ to $f'$ and $\mu^*(v)\ge2-\frac{5}{4}-\frac{3}{4}=0$. If $v$ is not incident with a 4-face, then by Lemma~\ref{3-face-property},  $v$ cannot be adjacent two pendant 3-faces. In this case, $\mu^*(v)\ge2-\frac{5}{4}-\frac{1}{2}>0$. Thus, we may assume that $v$ is not incident with a 3-face. Since $v$ is rich, $v$ is incident with at most 3 4-faces. If $v$ is incident with a $(3,3,4,4^+)$-face, then by By Lemma~\ref{lem3.6b}(3), $v$ is not incident to other $(3,3,4,4^+)$-face. By (R1.1.1), $\mu^*(v)\ge2-1-2\cdot\frac{1}{2}=0$. If $v$ is not incident with a $(3,3,4,4^+)$-face, then by (R1.1.1), $\mu^*(v)\ge2-3\cdot\frac{2}{3}=0$.

Next we assume that $v$ is poor. We distinguish  the following two cases.% and will show that $\mu^*(v)\geq 0$.

\smallskip

\n{\bf Case 1.} $N(v)$ has at least two vertices incident to triangles.

\smallskip

 Assume that $N(v)$ has at least three vertices incident to triangles, without loss of generality, that each of $v_1, v_2, v_3$ is incident with a triangle. Since $G$ contains no 5-cycle, $d(v_i)\geq 5$ for $i\in [3]$. By Lemma \ref{claim4.1}(4), $v_i$ for $i\in [3]$ gives at least $1$ to each incident $4$-face. By (R1.1.2), $v$ gives $0$ to $f_1$ and $f_2$, and at most $1$ to $f_3$ and $f_4$, respectively. Thus, $\mu^*(v)=2-1\cdot2=0$.  Thus, we assume that $N(v)$ has exactly two vertices incident with triangles.

First let the two vertices be $v_1$ and $v_2$.   By Lemma~\ref{claim4.1}(4), $f_1$ gets at least $2$ from $v_1$ and $v_2$.  By (R1.1.2), $v$ gives $0$ to $f_1$. Since only each of $v_1$ and $v_2$ is incident with a 3-face,   $v$ is $(v_1,v_3)$-behaved and $(v_2,v_4)$-behaved. By Lemma \ref{lem3.6b}(4) none of $f_i$ with $i\in [4]$ is a $(3,3,4,4^+)$-face. Thus $v$ gives at most $\frac{2}{3}$ to each of $f_2, f_3$ and $f_4$ by Lemma~\ref{claim4.1}(5). Thus,  $\mu^*(v)\ge2-3\cdot \frac{2}{3}=0$.

Then, by symmetry let the two vertices be $v_1$ and $v_3$.  Since $G$ has no $5$-cycle, $d(v_1)\geq 5$ and $d(v_3)\geq 5$. It follows that none of $f_i$ for $i\in [4]$ is a $(3,4,4,4)$-face.   If none of them is a $(3,3,4,4^+)$-face, then by Lemma~\ref{claim4.1} (5), $v$ gives at most $\frac{1}{2}$ to each $f_i$.  Thus, $\mu^*(v)\ge2-4\cdot\frac{1}{2}=0$.  So we may assume that $f_1$ is a $(3,3,4,4^+)$-face, i.e., $d(u_1)=d(v_2)=3$.  By Lemma~\ref{no-333-path}, $d(u_2)\ge 4$.  By Lemma~\ref{lem3.6b}(1) and (2), $d(u_3), d(u_4)\ge 4$. This implies that only one of $f_i$, where $i\in [4]$, is a $(3,3,4,4^+)$-face.

Let $d(v_4)\ge 4$.  By Lemma~\ref{claim4.1} (5), $v$ gives at most $\frac{1}{4}$ to each of $f_3$ and $f_4$, at most $1$ to $f_1$ and $\frac{1}{2}$ to $f_2$. Thus, $\mu^*(v)\geq 2-1-\frac{1}{2}-2\cdot\frac{1}{4}=0$.

Let $d(v_4)=3$.   By Lemma~\ref{claim4.1}(5), $v$ gives at most $\frac{1}{2}$ to $f_2$ and $f_3$, respectively.
If $d(v_1)=5$, then $u_4$ is rich since $G$ is 5-cycle free. By Lemma~\ref{claim4.1}(2), $u_4$ gives at least $\frac{1}{2}$ to $f_4$.  Note that the $5$-vertex $v_1$ is incident to a $3$-face and two $4$-faces, $d(v_2)=d(v_4)=3$ and at most one vertex in $\{u_1,v,u_4\}$ is incident with a triangle. By Lemma~\ref{3-face-and-4-face} (2),  the triangle incident with $v_1$ cannot be a $(3,4^-,5)$-face. By Lemma~\ref{claim4.1}(4), $v_1$ gives at least $\frac{5}{4}$ to each of $f_1$ and $f_4$. Thus,  $v$ gives at most $2-\frac{5}{4}-\frac{1}{2}=\frac{1}{4}$ to $f_4$ and $2-\frac{5}{4}=\frac{3}{4}$ to $f_1$.  If $d(v_1)\ge 6$, then by Lemma~\ref{claim4.1} (4), $v_1$ gives at least $\frac{4}{3}$ to each of $f_1$ and $f_4$. Thus, $v$ gives at most $2-\frac{4}{3}=\frac{2}{3}$ to $f_1$ and at most $\frac{2-\frac{4}{3}}{2}=\frac{1}{3}$ to $f_4$. Therefore, $\mu^*(v)\ge 2-2\cdot \frac{1}{2}-\max\{\frac{3}{4}+\frac{1}{4}, \frac{2}{3}+\frac{1}{3}\}=0$.

\smallskip

\n{\bf Case 2.}  $N(v)$ has at most one vertex incident with a triangle.

\smallskip

In this case, $v$ is $(v_1,v_3)$-behaved and $(v_2,v_4)$-behaved. It follows by Lemma \ref{lem3.6b}(4) that no $4$-faces incident to $v$ is a $(3,3,4,4^+)$-face. On the other hand,  if $v$ is not incident to a $(3,4,4,4)$-face, then by Lemma~\ref{claim4.1}(5), $v$ gives at most $\frac{1}{2}$ to each incident $4$-face. Thus $\mu^*(v)\ge2-4\cdot\frac{1}{2}=0$. Therefore, we may assume that $v$ is incident to a $(3,4,4,4)$-face, by symmetry, say $f_1$ such that $d(u_1)=3$ or $d(v_2)=3$.
\smallskip

\n{\bf Claim.} We may assume that none of $f_2,f_3,f_4$ is a $(3,4,4,4)$-face.

\smallskip

\n{\em Proof of Claim.} We may assume that $d(v_2)=3$.  For otherwise, let $d(u_1)=3$. Then by Lemma \ref{lem3.6b}(1) and (2), $d(u_i)\ge4$ for $i=2,3,4$. Since $d(v_1)=4$ and $v_1$ is $(u_1, u_4)$-behaved, by Lemma \ref{lem3.13}, $d(v_3)\ge4$. Similarly, $d(v_2)=4$ and $v_2$ is $(u_1, u_2)$-behaved implies that $d(v_4)\ge4$. Thus,  each $f_i$ is a $(4,4^+,4^+,4^+)$-face for $i=2,3,4$.

By Lemma \ref{4-face-property}(2), $d(v_3)\ge4$  and by Lemma \ref{lem3.13} $d(v_4)\ge4$. Moreover, since $d(v_2)=3$ and $v$ is a poor $4$-vertex and $(v_1,v_3)$-behaved, by Lemma \ref{lem3.13} either $d(v_4)=4$ and $v_4$ is not $(u_3,u_4)$-behaved or $d(v_4)\ge5$. It follows that none of $f_3$ and $f_4$ is a $(3,4,4,4)$-face. We suppose that $f_2$ is a $(3,4,4,4)$-face and will show that $\mu^*(v)\ge0$.
%Now suppose otherwise that $f_2$ is  a $(3,4,4,4)$-face and show that $\mu^*(v)\ge0$.

Since $v$ is $(v_2, v_4)$-behaved, and $d(v_1)=d(v)=d(v_3)=4$,  by Lemma~\ref{lem3.13},  $v_1$ is not $(u_1,u_4)$-behaved or $v_3$ is not $(u_2,u_3)$-behaved.  By symmetry, we assume that $v_3$ is not $(u_2,u_3)$-behaved. This means that each of $\{u_2, u_3\}$ is incident to a triangle. So $f_3$ is a $(4,4^+,4^+,4^+)$-face that is adjacent to a triangle. So by Lemma \ref{claim4.1}(5) $v$ gives at most $\frac{1}{4}$ to $f_3$. As  $d(u_2)=4$ and $u_2$ is incident to a triangle, by Lemma~\ref{claim4.1}(1) $u_2$ gives at least $\frac{3}{4}$ to $f_2$ and $v_3$ has at most three incident $4$-faces. %Now we show that $v_3$ is not incident to a $(3,3,4,4^+)$-face. Obviously, $f_2$ and $f_3$ are not $(3,3,4,4^+)$-face. So we suppose $v_3$ is incident to a $(3,3,4,4^+)$-face, that is $xyv_3u_3$ with $d(x)=d(y)=3$. Note that $y$ is not incident to a triangle by Proposition \ref{prop3.1}(c). Thus at most one vertex from each of $\{v,u_2\}$ and $\{u_3,y\}$ is incident to a triangle. Then by Lemma~\ref{lem3.6b}(2), $d(v_2)\ge4$ or $d(x)\ge4$, a contradiction. Thus $v_3$ is not incident to $(3,3,4,4^+)$-face.
By Claim \ref{claim4.1}(3) $v_3$ gives at least $\frac{2}{3}$ to $f_2$. So by (R1.1.2), $v$ gives at most $2-\frac{3}{4}-\frac{2}{3}=\frac{7}{12}$ to $f_2$. Note that $v$ gives at most $\frac{2}{3}$ to $f_1$ and $\frac{1}{2}$ to $f_4$ by Lemma~\ref{claim4.1}(5). Thus $\mu^*(v)\ge2-\frac{2}{3}-\frac{7}{12}-\frac{1}{4}-\frac{1}{2}=0$.
This proves our claim.

\smallskip
Now we are ready to complete our proof. %Recall that $f_1$ is a $(3, 4, 4, 4)$-face. By symmetry, we assume that $d(v_2)=3$ or $d(u_1)=3$. By the Claim, none of $f_2$ and $f_3$ is  $(3,4,4,4)$-face.
By Lemma~\ref{claim4.1}, $v$ gives at most $\frac{2}{3}$ to $f_1$ and $\frac{1}{2}$ to each of $f_2$ and $f_3$.
In order to show that $\mu^*(v)\geq 0$, we just need to show that $v$ gives at most $\frac{1}{3}$ to $f_4$.

We may assume that $d(v_4)\ge 5$.  Note that $d(v_2)=3$, or if $d(u_1)=3$, then $u_1$ is not incident with a triangle by Proposition \ref{prop3.1}(c) and hence  $v_2$ is $(u_1,u_2)$-behaved. It follows  by Lemma \ref{lem3.13} that $d(v_4)=4$ and $v_4$ is not $(u_3,u_4)$-behaved or $d(v_4)\ge5$. But in the former case, that means  both $u_3$ and $u_4$ are incident to triangles. By Lemma~\ref{claim4.1}(5), $v$ gives at most $\frac{1}{4}$ to $f_4$. Therefore, we may assume the latter is true, that is, $d(v_4)\ge 5$.

% In the latter case, by our assumption that $f_1=(3, 4,4,4)$ with $d(v_2)=3$ or $d(u_1)=3$, $d(v_1)=4$. Thus, let $v_1'\in N(v_1)$ and $f_5=v_1v_1'u_1'u_1$ and $f_6=v_1v_1'u_4'u_4$. We now consider the following two cases when $v_1$ is a poor or rich $4$-vertex.

Assume first that $v_1$ is a poor $4$-vertex. Then the four $4$-faces incident to $v_1$ are $f_1$, $f_4$, $f_5$ and $f_6$, where $f_5=v_1v_1'u_1'u_1$ and $f_6=v_1v_1'u_4'u_4$. As $d(u_1)=3$ or $d(u_1)=4$ and $u_1$ is $(u_1',v_2)$-behaved,  and $v_1$ is $(v,v_1')$-behaved,  by Lemma \ref{lem3.13} $d(u_4)=4$ and $u_4$ is not $(u_4',v_4)$-behaved or $d(u_4)\ge5$. In the former case, since $v_4$ is incident with a triangle and $d(v_4)\geq 5$, by Lemma~\ref{claim4.1}(4), $f_4$ gains at least $1$ from $v_4$; If $u_4$ is a poor $4$-vertex, By (R1.1.2), $v$ gives at most $\frac{2-1}{3}=\frac{1}{3}$ to $f_4$; If $u_4$ is rich, $u_4$ gives at least $\frac{1}{2}$ to $f_4$, thus by (R1.1.2), $v$ gives at most $\frac{2-1-\frac{1}{2}}{2}=\frac{1}{4}$ to $f_4$.   In the latter case, if at least one of $u_4$ and $v_4$ is a rich $5$-vertex or $6^+$-vertex, then by (R1.2.2) and Lemma\ref{claim4.1}(4) $v$ gives at most $\frac{2-1-\frac{1}{2}}{2}=\frac{1}{4}$ to $f_4$;  thus, we may assume that both $u_4$ and $v_4$ are poor $5$-vertices, but it follows that  $f_4$ is a special $(4,4,5,5)$-face, and by (R1.2.2)and (R1.1.2), $v$ gives $\frac{2-\frac{3}{4}\cdot2}{2}=\frac{1}{4}$ to $f_4$.

Now we assume that $v_1$ is a rich $4$-vertex. Then $v_1$ is incident to at most three $4$-faces.

We first show that $v_1$ cannot be incident to a $(3,3,4,4^+)$-face. Suppose otherwise that $v_1$ is incident to such 4-face. Note that $f_1$ and $f_4$ are not $(3,3,4,4^+)$-face.  Thus assume that $v_1$ is incident to a $(3,3,4,4^+)$-face $f_5$ that share an edge with $f_1$ or $f_4$. Let $N(v_1)=\{u_1,v,u_4,v_1'\}$. If $d(u_1)=3$, then $v_1$ is $(u_1,u_4)$-behaved and $(v,v_1')$-behaved, thus by Lemma \ref{lem3.6b}(4) $f_5$ cannot be a $(3,3,4,4^+)$-face, a contradiction.  If $d(v_2)=3$, then $d(u_1)=4$, thus by Lemma \ref{lem3.6b}(1) and (2), $f_5$ cannot be a $(3,3,4,4^+)$-face, a contradiction.

Thus by Lemma \ref{claim4.1}(2), $v_1$ gives at least $\frac{2}{3}$ to $f_4$. Now we consider the degree of $u_4$. Recall that $d(v_4)\geq 5$ and $v$ is a poor 4-vertex. If $u_4$ is a $3$-vertex, then  by Lemma~\ref{claim4.1}(4) or (R1.2.2), $v_4$ gives at least $1$ to $f_4$, thus by (R1.1.2), $v$ gives at most $2-1-\frac{2}{3}=\frac{1}{3}$ to $f_4$.  If $u_4$ is a rich $4$-vertex or $d(u_4)\ge5$, then by Lemma~\ref{claim4.1}(2)(4) and (R1.2.2), $v$ gives at most $2-\frac{1}{2}\cdot2-\frac{2}{3}=\frac{1}{3}$ to $f_4$. Finally let $u_4$ be a poor $4$-vertex. If  $v_4$ is a poor 5-vertex,   then $f_4$ is a special $(4,4,4,5)$-face, thus by  (R1.2.2), $v_4$ gives $\frac{3}{4}$ to $f_4$;  If $v_4$ is not a poor $5$-vertex,  then  Lemma \ref{claim4.1}(4), $v_4$ gives at least $1$ to $f_4$. Thus, by (R1.1.2) $v$ gives at most $\max\{\frac{2-\frac{3}{4}-\frac{2}{3}}{2},\frac{2-1-\frac{2}{3}}{2}\}=\frac{7}{24}\le\frac{1}{3}$ to $f_4$.
\end{proof}

In order to prove that $5$-vertices have nonnegative charges (Lemma~\ref{le43}), we first handle two special cases in Lemmas~\ref{le431} and~\ref{le432}.

\begin{lem}
\label{le431}
Suppose that $v$ is a poor 5-vertex and $N(v)$ has no vertex incident to a triangle. If $f_i=u_iv_{i+1}vv_i$ is a $(3,4,5,4)$-face, then $\mu^*(v)\geq 0$.
\end{lem}
\begin{proof}
By symmetry, let $i=1$. First we show that none of $f_2$ and $f_5$ is a special $(4,4,5,5)$-or $(4,4,4,5)$-face. Suppose otherwise that by symmetry  $f_2$ is a special $(4,4,5,5)$- or $(4,4,4,5)$-face. By the definition of special $(4,4,5,5)$-or $(4,4,4,5)$-face, $v_2$ is poor  and $d(v_2)=d(u_2)=4$. By Lemma \ref{lem3.13}, $v_3$ must be incident to a triangle, a contradiction.  It follows that if $f_2$ (or $f_5$) is a $(4^+, 4^+, 4^+,5)$-face, then $v$ gives at most $\frac{1}{2}$ to it.

By Lemma \ref{lem3.14} (2), at most two vertices in $u_i$ with $i\in [5]$ are $3$-vertices. Since $d(u_1)=3$,  at most one of $u_2$ and $u_5$ is a $3$-vertex. By symmetry we consider the following two cases.

Assume first that $d(u_5)\ge4$ and $d(u_2)\ge4$. If $\min\{d(v_3), d(v_5)\}\ge4$, then $v$ gives at most $\frac{1}{2}$ to $f_2$ and $f_5$ and at most $1$ to each other incident $4$-face, thus $\mu^*(v)\ge4-\frac{1}{2}\cdot2-1\cdot3=0$. So we may assume by symmetry that $d(v_5)=3$.   By Lemma \ref{4-face-property}(2) $d(v_4)\ge4$.  Since $d(u_1)=3$, by Lemma \ref{lem3.14}(1), $d(u_4)\ge4$.    We claim that $d(u_3)\ge4$, for otherwise, since $d(v_5)=3$, by Lemma \ref{lem3.14}(3) $d(v_2)=4$ and $v_2$ is not $(u_1, u_2)$-behaved, or $d(v_2)\ge5$, which is  contrary to our assumption that $d(v_2)=4$ and $v_2$ is $(u_1, u_2)$-behaved(note that $u_1$ cannot be incident to a triangle).  Since $d(v_5)=3$, applying  Lemma \ref{lem3.14} (3) to $u_1$, we get $d(v_3)=4$ and $v_3$ is not $(u_2,u_3)$-behaved or $d(v_3)\ge 5$. In the former case,  $f_2$ is a $(4, 5, 4^+, 4)$-face  with $u_2$ incident to a triangle and $f_3$ is a $(4, 5, 4^+, 4^+)$-face with $u_3$ incident to a triangle, then by (R1.2.2), $v$ gives at most $\frac{1}{2}$ to each of $f_2$ and $f_3$, thus, $\mu^*(v)\geq 4-1\cdot3-2\cdot\frac{1}{2}=0$. Consider the latter case now.  As the argument above, $v$ gives at most $\frac{1}{2}$ to $f_2$. Note that $f_3$ is a $(4^+,4^+,5^+,5)$-face, so if $f_3$ is not special $(4,4,5,5)$-face, then by (R1.2.2) $v$ gives at most $\frac{1}{2}$ to $f_3$, and it follows that $\mu^*(v)\geq 0$; thus, we may assume that $f_3$ is a special $(4,4,5,5)$-face.  It follows that $v_4$ and $u_3$  are both poor 4-vertices. Since $v_3$ and $v_5$ are not incident to triangles, applying Lemma~\ref{lem3.13} to $v_4$, we have $d(u_4)\ge 5$,  so $f_4$ is a $(3,5,4,5^+)$-face. By (R1.2.2), $v$ gives at most $\frac{3}{4}$ to each of $f_3$ and $f_4$, so $\mu^*(v)\geq 4-2\cdot 1-2\cdot\frac{3}{4}-\frac{1}{2}=0$.

Assume now by symmetry that $d(u_5)=3$ and $d(u_2)\ge4$.  By Lemma \ref{lem3.14}(1), $d(v_5)\ge4$.   By Lemma~\ref{lem3.14} (2) $d(u_i)\ge4$ for $i=2,3,4$. Since $d(v_2)=4$ and $v_2$ is $(u_1,u_2)$-behaved,  by Lemma~\ref{lem3.14} (3) (with $i=5$), we get $d(v_4)\ge 5$ or $d(v_4)=4$ and $v_4$ is not $(u_3,u_4)$-behaved.  If $d(v_3)\ge4$, then $v$ gives at most $\frac{3}{4}$ to each of $f_3$ and $f_4$ and $\frac{1}{2}$ to $f_2$ by (R1.2.2), thus $\mu^*(v)\geq 0$. So let $d(v_3)=3$.  By Lemma~\ref{lem3.14} (3) (with $i=1$), we get $d(v_5)\ge5$ (since $d(u_5)=3$, $v_5$ is not $(u_4, u_5)$-behaved). Since $d(v_4)\ge 5$ or $d(v_4)=4$ but both $u_3$ and $u_4$ are incident to triangles,  $f_4$ is a $(4^+,4^+,5,5^+)$-face but not a special $(4,4,5,5)$-face and $f_3$ is a $(3,4^+,4^+,5)$-face but not a special $(3,4,4,5)$-face. Thus, by (R1.2.2), $v$ gives at most $\frac{1}{2}$ to $f_4$ and $\frac{3}{4}$ to each $f_3$ and $f_5$. So $\mu^*(v)\geq4-2\cdot 1-2\cdot\frac{3}{4}-\frac{1}{2}=0$.
\end{proof}

\begin{lem}
\label{le432}
Suppose that $v$ is a poor 5-vertex and $N(v)$ has no vertex incident to a triangle.  If $f_i=v_iu_iv_{i+1}v$  is a special $(3,4,4,5)$-face, then $\mu^*(v)\geq 0$.
\end{lem}
\begin{proof}
By symmetry, let $i=1$. By Lemma~\ref{le431} we may assume that $v$ is not incident to a $(3,4,5,4)$-face.  Since $d(v_1)=3$, by Lemma \ref{4-face-property}(2) $d(v_5)\ge4$. Since $f_1$ is a special $(3,4,4,5)$-face, $u_1$ is not incident to a triangle, thus at most one vertex in $\{u_1,v,u_2\}$ is incident to a triangle, so by applying Lemma \ref{lem3.6b}(1) to $v_2$, $d(v_3)\ge4$.

We may assume that $f_2$ is neither a special $(4,4,4,5)$-face nor a special $(4,4,5,5)$-face.  Suppose otherwise, then $v_2$ is poor and $d(v_2)=d(u_2)=d(u_1)=4$, and none of $v_1, v, v_3$ is incident to a triangle, a contradiction to Lemma \ref{lem3.13}.  It follows that if $f_2$ is a $(4,5,4^+,4^+)$-face, then by (R1.2.2), $v$ gives at most $\frac{1}{2}$ to $f_2$.

We may also assume that $d(u_5)\ge4$.   Suppose otherwise that $d(u_5)=3$.  By Lemma \ref{lem3.14} (1) $d(u_i)\ge4$ for $i=2,3,4$.  Since $f_1$ is a special $(3,4,4,5)$-face, $u_1$ is not incident with a triangle.  Applying Lemma~\ref{lem3.14} (3) to $u_5$, $d(v_4)\ge 4$, so by (R1.2.2), $v$ gives at most $\frac{3}{4}$ to each $f_3$ and $f_4$. Note that $v$ gives at most $1/2$ to $f_2$, as it is a $(4,5,4^+,4^+)$-face.  But now $\mu^*(v)\geq4-2\cdot 1-2\cdot\frac{3}{4}-\frac{1}{2}=0$.

Now we consider the following four cases depending on the degree of $u_3$ and $u_4$.

Let $d(u_3)=d(u_4)=3$. By Lemma~\ref{no-333-path} $d(v_4)\ge4$. By Lemma \ref{lem3.14}(3) (with $i=4$) $d(v_3)\ge5$ and (with $i=3$) $d(v_5)\ge5$.  It follows that for $i\in \{2,3,4,5\}$, $f_i$ is a $(3^+, 4^+, 5,5^+)$-face, so by (R1.2.2), $v$ gives at most $\frac{3}{4}$ to $f_i$. Thus, $\mu^*(v)\geq 4-1-4\cdot\frac{3}{4}=0$.

Let $d(u_3)=3$ and $d(u_4)\ge4$. Since $d(v_2)=4$ and $v_2$ is $(u_1, u_2)$-behaved, by Lemma \ref{lem3.14}(3) (with $i=3$),  either $d(v_5)=4$ and $v_5$ is not $(u_4,u_5)$-behaved or $d(v_5)\ge5$, then $f_4$ and $f_5$ are $(3,4^+, 4^+, 5)$-faces but not special $(3,4,4,5)$-faces, so by (R1.2.2),  $v$ gives at most $\frac{3}{4}$ to each of $f_4$ and $f_5$.   If $d(u_2)\ge4$, then $v$ gives at most $\frac{1}{2}$ to $f_2$ which is a $(4,5,4^+,4^+)$-face, so $\mu^*(v)\geq4-2\cdot 1-2\cdot\frac{3}{4}-\frac{1}{2}=0$. Thus,  we assume that $d(u_2)=3$. As $f_2$ cannot be a $(3,4,5,4)$-face,  $d(v_3)\ge 5$, so $f_2$ and $f_3$ are $(3, 4^+, 5, 5^+)$-faces, then by (R1.2.2),  $v$ gives at most $\frac{3}{4}$ to each of $f_3$ and $f_2$. We conclude that  $\mu^*(v)\geq4-1-4\cdot\frac{3}{4}=0$.

Let $d(u_3)\ge4$ and $d(u_4)=3$. By Lemma \ref{lem3.14}(3) (with $i=4$),  $d(v_3)=4$ and $v_3$ is not $(u_2,u_3)$-behaved or $d(v_3)\ge5$.  In the former case, $f_2$ is a $(4, 5, 4, 4^+)$-face with $u_2$ incident to a triangel and $f_3$ is a $(3^+, 4^+, 4, 5)$-face with $u_3$ incident to a triangle; In the latter case, each of $f_2$ and $f_3$ is a $(3^+, 4^+,5,5^+)$-face; so by (R1.2.2)  $v$ gives at most $\frac{3}{4}$ to each of $f_2$ and $f_3$. If $d(u_2)=3$, then by Lemma~\ref{lem3.14} (3) (with $i=2$), $d(v_4)\ge5$,  thus $f_2$ is a $(3, 4, 5, 5^+)$-face, $f_3$ is a $(4^+, 5^+, 5, 5^+)$-face and $f_4$ is a $(3, 4^+, 5, 5^+)$-face, so by (R1.2.2), $v$ gives at most $\frac{1}{2}$ to $f_3$ and at most $\frac{3}{4}$ to each $f_2$ and $f_4$, therefore, $\mu^*(v)\geq4-2\cdot 1-2\cdot\frac{3}{4}-\frac{1}{2}=0$.   Now we assume that $d(u_2)\ge4$.   If $d(v_5)=4$, then at most one of $\{u_5, v, u_4\}$ is incident with a triangle,  so by applying Lemma~\ref{lem3.6b} (1) to $v_5$, we have $d(v_4)\ge 4$; as $v$ is not incident to $(3,4,5,4)$-faces, we further conclude that $d(v_4)\ge 5$. Now, $f_2$ is a $(4, 5, 4^+, 4^+)$-face, $f_3$ is a $(4^+, 4^+, 5, 5^+)$-face, and $f_4$ is a $(3, 4, 5, 5^+)$-face, so by (R1.2.2), $v$ gives at most $\frac{1}{2}$ to $f_2$ and $\frac{3}{4}$ to each of $f_3$ and $f_4$. It follows that $\mu^*(v)\geq4-2\cdot 1-2\cdot\frac{3}{4}-\frac{1}{2}=0$.   If $d(v_5)\ge 5$, then by (R1.2.2), $v$ gives at most $\frac{3}{4}$ to each of $f_4$ and $f_5$, and  gives at most $\frac{1}{2}$ to $f_2$ which is a $(4,5,4^+,4^+)$-face. We conclude that $\mu^*(v)\geq4-2\cdot 1-2\cdot\frac{3}{4}-\frac{1}{2}=0$.

We are left to consider the case that $d(u_3)\ge4$ and $d(u_4)\ge4$.

 Assume first that $d(u_2)\ge4$. Note that $v$ gives $\frac{1}{2}$ to $f_2$ which is a $(4,5,4^+,4^+)$-face.  If $d(v_4)\ge4$, then each of $f_3$ and $f_4$ is a $(4^+,4^+,4^+,5)$-face, so by (R1.2.2), $v$ gives $\frac{3}{4}$ to each of $f_3$ and $f_4$, it follows that $\mu^*(v)\geq4-2\cdot 1-2\cdot\frac{3}{4}-\frac{1}{2}=0$. So let $d(v_4)=3$.  Then both $f_4$ and $f_5$ are $(3, 4^+,4^+,5)$-faces. If $d(v_5)=4$ and $v_5$ is $(u_4,u_5)$-behaved, then at most one of $\{u_4, u_5, v\}$ is incident with a triangle, and $d(v_1)=d(v_4)=3$, a contradiction to Lemma~\ref{lem3.6b}(1).  This means either $d(v_5)=4$ and both $u_4$ and $u_5$ are incident to triangles or $d(v_5)\ge5$. Thus, none of $f_4$ and $f_5$ is a special $(3, 4, 4, 5)$-face. By (R1.2.2), $v$ gives at most $\frac{3}{4}$ to each of $f_4$ and $f_5$. Thus, we also have $\mu^*(v)\geq4-2\cdot 1-2\cdot\frac{3}{4}-\frac{1}{2}=0$.

 Thus, we may assume that $d(u_2)=3$. By Lemma~\ref{lem3.14} (3) (with $i=2$), either $d(v_4)=4$ and $v_4$ is not $(u_3,u_4)$-behaved or $d(v_4)\ge5$. In the former case, both $f_3$ and $f_4$ are $(4, 4^+, 4^+,5)$-faces, but none of them is a special $(4,4,4,5)$- or $(4,4,5,5)$-face, since $u_3$ and $u_4$ are incident with triangles; by (R1.2.2), $v$ give at most $\frac{1}{2}$ to $f_3$ and $f_4$, thus, $\mu^*(v)\geq4-3-2\cdot\frac{1}{2}=0$. So consider the latter case that $d(v_4)\ge 5$.  We claim that $f_3$ is not a special $(4,4,5,5)$-face, for otherwise, $d(v_3)=d(u_3)=4$ and $v_3$ is poor,  but none of $v$ and $v_4$ is incident to triangles, and $d(u_2)=3$, a contradiction to Lemma \ref{lem3.13}. It follows by (R1.2.2) that $v$ gives at most $\frac{1}{2}$ to $f_3$.   If $f_4$ is not a special $(4,4,5,5)$-face,  then by (R1.2.2), $v$ gives at most $\frac{1}{2}$ to $f_4$, which implies that $\mu^*(v)\geq4-3-2\cdot\frac{1}{2}=0$.  Thus, we assume that $f_4$ is a special $(4,4,5,5)$-face. It follows that $d(u_4)=d(v_5)=4$ and $v_5$ is poor. By Lemma \ref{lem3.13}, $d(u_5)\ge5$. It follows that $f_5$ is a $(3, 5,4, 5^+)$-face. By (R.1.2.2), $v$ gives at most $\frac{3}{4}$ to each of $f_5$ and $f_4$.  Therefore,  $\mu^*(v)\geq4-2\cdot 1-2\cdot\frac{3}{4}-\frac{1}{2}=0$.
\end{proof}

\begin{lem}
\label{le43}
Each $5$-vertex $v\in int(C_0)$ has nonnegative final charge.
\end{lem}
\begin{proof}
%  Assume that $v$ is rich. If $v$ is incident with a 3-face and adjacent to 3 pendant 3-face,  then by (R1.2.1), $\mu^*(v)=4-2- 3\cdot\frac{1}{2}>0$ if one incident triangle is $(3, 4^-,5)$-face and $\mu^*(v)=4-\frac{3}{2}-3\cdot\frac{1}{2}>0$. If $v$ is incident with 4-faces, then by (R1.2.1), $\mu^*(v)=0$.c

If $v$ is rich, then by (R1.2.1), $v$ gives at most $2+max\{1\cdot2, 3\cdot \frac{1}{2}\}=4$ to incident triangles and pendant $3$-faces and incident $4$-faces, thus its final charge must be nonnegative.  Thus, we may assume that $v$ is poor.

We may further assume that some vertex in $N(v)$ is incident to a triangle. Suppose otherwise. By Lemmas~\ref{le431} and ~\ref{le432}, we may assume that $v$ is not incident to a $(3,4,5,4)$-face or a special $(3,4,4,5)$-face. If $v$ is not incident to a $(3,3,5,4^+)$-face, then by (R1.2.2), $\mu^*(v)\geq 4-5\cdot\frac{3}{4}>0$, so by symmetry, we assume that  $f_1=u_1v_1vv_2$ is a  $(3,3,5,4^+)$-face. By Lemma \ref{4-face-property}(2), $d(v_5)\ge4$.  For $i\in \{2, 3, 4, 5\}$,  $d(u_i)\ge4$ by Lemma \ref{lem3.14} (1), then $f_i$ cannot be a $(3,3,5,4^+)$-face, so by (R1.2.2), $v$ gives at most $3/4$ to $f_i$. It follows that $\mu^*(v)\ge4-1-\frac{3}{4}\times4=0$.

Now we consider the following two cases.

\smallskip

\n{\bf Case 1.} $N(v)$ has at least two vertices incident to triangles.

\smallskip

If $v_i$ and $v_{i+1}$ for some $i\in [5]$ are incident to triangles,  then $f_i$ is rich and by (R1.2.2) $v$ gives $0$ to $f_i$ and at most $1$ to each other $4$-face. Thus,  $\mu^*(v)\geq 4-4=0$. We assume, without loss of generality,  that $v_1$ and $v_3$ are incident with triangles.   If $d(v_2)\ge 4$, then by (R1.2.2), $v$ gives at most $\frac{1}{2}$ to each of $f_1$ and $f_2$, and gives at most 1 to each of $f_3, f_4$ and $f_5$, so $\mu^*(v)\ge 4-3\cdot 1-2\cdot\frac{1}{2}=0$. Thus,  we may assume that $d(v_2)=3$. If $\min\{d(u_1), d(u_2)\}\geq 4$, then
each of $f_1$ and $f_2$ is a $(3, 4^+, 5, 5^+)$-face but not a special $(3,4,5,5)$-face, so by (R1.2.2), $v$ gives at most $\frac{1}{2}$ to each of $f_1$ and $f_2$, therefore, $\mu^*(v)\ge 4-3\cdot 1-2\cdot\frac{1}{2}=0$. Thus, by symmetry, assume that $d(u_1)=3$.  By Lemma \ref{no-333-path} $d(u_2)\ge 4$. Note that $v$ gives at most $\frac{1}{2}$ to $f_2$. If one of  $f_3$ and $f_5$, say $f_3$,  is not $(3,3,5,5^+)$-face, then $f_3$ is a $(3,4,5,5^+)$-face, so by (R1.2.2), $v$ gives $\frac{1}{2}$ to $f_3$, therefore, $\mu^*(v)\geq 0$. Thus, we may assume that $f_3$ and $f_5$ are $(3,3,5,5^+)$-faces.  It follows that $d(v_4)=d(v_5)=3$. By Lemma~\ref{4-face-property}(2), each of $u_4$ and $v$ is incident with a triangle, a contradiction.

\smallskip

\n{\bf Case 2.} $N(v)$ has exactly one vertex  incident to a triangle.

\smallskip

We assume, without loss of generality, that $v_1$ is incident with a triangle. If neither $f_1$ nor $f_5$ is a $(3,3,5,5^+)$-face, then by (R1.2.2), $v$ gives at most $\frac{1}{2}$ to each of them. This implies that  $\mu^*(v)\geq 4-3\cdot 1-2\cdot \frac{1}{2}=0$. Thus, by symmetry we may assume that $f_5$ is a $(3,3,5,5^+)$-face.  It follows that $d(u_5)=d(v_5)=3$.  By Lemma~\ref{lem3.14}(1), $d(u_i)\ge 4$ for $i\in [4]$. By Lemma~\ref{4-face-property}(2), $d(v_4)\ge 4$.  By (R1.2.2) $v$ gives at most $\frac{1}{2}$ to $f_1$.

We may assume that $d(v_4)=4$, for otherwise, both $f_3$ and $f_4$ are $(3^+, 4^+, 5^+, 5)$-faces, thus by (R1.2.2), $v$ gives at most $\frac{3}{4}$ to each of them, so $\mu^*(v)\geq 4-\frac{1}{2}-2\cdot\frac{3}{4}-2\cdot1=0$.   By applying Lemma~\ref{lem3.6b} (1) on $4$-vertex $v_4$,  we get either both $u_3$ and $u_4$ are incident to triangles or $d(v_3)\ge 4$. In the former case, none of $f_3$ and $f_4$ is a special $(3, 4, 4,5)$-face, thus by (R1.2.2), $v$ gives at most $\frac{3}{4}$ to each of them, so $\mu^*(v)\geq 4-\frac{1}{2}-2\cdot\frac{3}{4}-2\cdot1=0$. Consider the latter case which $d(v_3)\ge 4$. If  $f_3$ is neither a special $(4,4,4,5)$-face nor a special $(4,4,5,5)$-face, then by (R1.2.2), $v$ gives at most $\frac{1}{2}$ to $f_3$, so $\mu^*(v)\geq 4-3\cdot 1-2\cdot \frac{1}{2}=0$.  Thus, we may assume that $f_3$ is a special $(4,4,4,5)$-face or a special $(4,4,5,5)$-face. By (R1.2.2), $v$ gives at most $\frac{3}{4}$ to $f_3$.   By the definition of special $(4,4,4,5)$-face or $(4,4,5,5)$-face, $v_4$ is poor and $d(v_4)=d(u_3)=4$.  Note that no vertex in $\{v_3,v,v_5\}$ is incident to a triangle.  By Lemma \ref{lem3.13}, $d(u_4)\ge5$. So $f_4$ is a $(3,5,4,5^+)$-face and by (R1.2.2) $v$ gives at most $\frac{3}{4}$ to $f_4$. Thus, $\mu^*(v)\geq 4-(2\cdot 1+2\cdot \frac{3}{4}+\frac{1}{2})=0$.
\end{proof}

Now we consider the case $v\in C_0$.

\begin{lem}
Each $v\in C_0$ has nonnegative final charge.
\end{lem}

\begin{proof}

We consider the following cases according to the degree of $v$. For $l=3,4$, by Lemma \ref{lem3.4} each $l$-face $f$ in $G$ satisfies that $|b(f)\cap C_0|\le2$ and furthermore, when $|b(f)\cap C_0|=2$, $f$ and $C_0$ share a common edge.

\begin{enumerate}[(1)]
\item $d(v)=2$.  By (R3),  $\mu^*(v)=2\times 2-6+2=0$.

\item $d(v)=3$.  Then $v$ could be incident with at most one triangle from $F_3''$ or has at most one pendant $3$-face from $F_3'$.  By (R2) and (R3),  $\mu^*(v)\ge 2\times3-6-\frac{3}{2}+\frac{3}{2}=0$.

\item $d(v)=4$. Assume first that $v$ is incident with a $3$-face $f$. If $f\in F_3'$, then by (R2) and (R3), $\mu^*(v)=2-3+1=0$. If $f\in F_3''$, then it could be incident to at most one $4$-face from $F_4''$ or adjacent to at most one pendent $3$-face from $F_3$. By (R2) and (R3), $\mu^*(v)\ge2-\frac{3}{2}-1+1=\frac{1}{2}>0$. Thus, we may assume $v$ is not incident to a $3$-face. By  Lemma~\ref{4-face-property} (1), $v$ is not incident face from $F_4'$. Thus, we assume that $v$ is incident with $k\le2$ $4$-faces from $F_4''$. Then $v$ is adjacent to at most $2-k$ pendent $3$-faces from $F_3$. By (R2) and (R3),  $\mu^*(v)\ge2-k-\frac{1}{2}(2-k)\ge0$.

\item $d(v)=k\ge 5$.  If $v$ is not incident with any $3$-face, then by Lemma~\ref{4-face-property}, $v$ is not incident face from $F_4'$, so by (R2), $\mu^*(v)\ge 2k-6-1\cdot (k-2)\ge1>0$. Thus, we first assume that  $v$ is incident with a face from $F_3'$.  Let $s$ be the number of $4$-faces in $F_4'$ incident with $v$.  If $s=0$, then by (R2), $\mu^*(u)\ge2k-6-(k-4)-3\geq 0$; and if $s\ge1$, then $s\le k-5$. By (R2), $\mu^*(v)\ge2k-6-3-\frac{3}{2}s-(k-s-4)=k-\frac{1}{2}s-5\ge \frac{1}{2}$.
  Next, we assume that $v$ is incident with a face from $F_3''$. If $s=0$,  then by (R2), $\mu^*(v)\ge2k-6-(k-3)-\frac{3}{2}\ge\frac{1}{2}$; if  $s\ge1$, then  $s\le k-4$. By (R2), $\mu^*(v)\ge2k-6-\frac{3}{2}s-\frac{3}{2}-(k-s-3)=k-\frac{1}{2}s-\frac{9}{2}\ge \frac{1}{2}s-\frac{1}{2}\ge0$.

\end{enumerate}
\end{proof}

Then we consider faces.  As $G$ contains no $5$-faces, and $6^+$-faces other than $C_0$ are not involved in the discharging procedure, we only need to show that $C_0$, and $3$-faces and $4$-faces other than $C_0$ have nonnegative charges.

\smallskip

\begin{lem}
Each $3$-face $f\not=C_0$ has nonnegative final charge.
\end{lem}

\begin{proof}
Note that $f$ has initial charge $3-6=-3$. By Lemma \ref{lem3.4}  $|b(f)\cap C_0|\le2$. If $|b(f)\cap C_0|=1$, then by (R2), $\mu^*(f)\ge-3+3=0$;  if $|b(f)\cap C_0|=2$, then by (R2), $\mu^*(f)\ge-3+\frac{3}{2}\times 2=0$.  Thus, we may assume that  $b(f)\cap C_0=\emptyset$.   Let $f=uvw$ with corresponding degrees $(d_1, d_2, d_3)$. Let $x'$ be the pendant neighbor of $x$ on a $3$-face f. By Lemma \ref{3-face-property} (1), we only need to check the following cases:

\begin{enumerate}[(1)]
\item $f$ is a $(3,3,5^+)$-face. By Lemma \ref{3-face-property} (2), $u'$ and $v'$ are either on $C_0$ or have degree at least $4$. By (R1.1.1) and (R1.2.1), $f$ receives $\frac{1}{2}$ from each of $u'$ and $v'$. By (R1.2.1) and (R1.3), $f$ receives $2$ from $w$. Thus, $\mu^*(f)=-3+\frac{1}{2}\times2+2=0$.

\item $f$ is a $(3,4,4)$-face. By Lemma \ref{3-face-property} (3), Then $u'$ is either on $C_0$ or has degree at least $4$. By (R1.1.1) and (R1.2.1), $f$ receives $\frac{1}{2}$ from $u'$. By (R1.1.1), $f$ receives $\frac{5}{4}$ from each of $v$ and $w$. Thus, $\mu^*(f)=-3+\frac{5}{4}\times2+\frac{1}{2}=0$.

\item $f$ is a $(3,4,5)$-face. By (R1.1.1) and (R1.2.1), $f$ receives $1$ from $v$ and $2$ from $w$. Thus, $\mu^*(f)=-3+1+2=0$.

\item $f$ is a $(3,5,5)$-face. By (R1.2.1), $f$ receives $\frac{3}{2}$ from each of $v$ and $w$. Thus, $\mu^*(f)\ge-3+\frac{3}{2}\times2=0$.

\item $f$ is a $(3,4^+,6^+)$-face. By (R1.1.1), (R1.2.1) and (R1.3), $f$ receives at least $1$ from $v$ and $2$ from $w$. Thus, $\mu^*(f)\ge-3+1+2=0$.

\item $f$ is a $(4^+,4^+,4^+)$-face. By (R1.1.1),(R1.2.1) and (R1.3), $f$ receives at least $1$ from each of $u,v$ and $w$. Thus, $\mu^*(f)\ge-3+1\times3=0$.
\end{enumerate}
\end{proof}

\begin{lem}
Each $4$-face $f\not=C_0$ has nonnegative final charge.
\end{lem}

\begin{proof}
Let $f=uvwx$ with corresponding degrees $(d_1,d_2, d_3,d_4)$.  Note that $f$ has initial charge $4-6=-2$. By Lemma \ref{lem3.4}  $|b(f)\cap C_0|\le2$. If $|b(f)\cap C_0|=1$, say $u\in b(f)\cap C_0$, then by (R2), $u$ gives $\frac{3}{2}$ to $f$;  By Lemma \ref{4-face-property} each of $u$ and $w$ is incident to a triangle, so $d(w)\ge4$ and by (R1.1.1),(R1.2.1) and (R1.3), $w$ gives at least $min\{\frac{3}{4}, 1, \frac{6-2}{3}\}=\frac{3}{4}$ to $f$;  So $\mu^*(f)\ge-2+\frac{3}{2}+\frac{3}{4}>0$. If $|b(f)\cap C_0|=2$, then by (R2), $\mu^*(f)\ge-2+1\times2=0$.  So we now assume that $b(f)\cap C_0=\emptyset$. If some vertex on $b(f)$ is poor $4$-vertex, then by (R1.1.2), the poor $4$-vertex will give enough charges to $f$ to make its final charge to be $0$. So we assume that each $4$-vertex on $b(f)$ is rich. By Lemma \ref{4-face-property} (2), we only need to consider the following $4$-faces.

\begin{enumerate}[(1)]
\item $f$ is a $(3,3,4^+,4^+)$-face. By (R1.1.1), (R1.2) and (R1.3), each $4^+$-vertex gives at least $1$ to $f$. Thus, $\mu^*(f)\ge-2+1\times2=0$.

\item $f$ is a $(3,4^+,3,4^+)$-face.  Then by Lemma~\ref{4-face-property}(2), and both $v$ and $x$ are incident to a triangle. By Lemma~\ref{claim4.1} (1), $f$ receives at least $1$ from each of $v$ and $x$, so $\mu^*(f)\ge-2+1\times2=0$.

\item $f$ is a $(3,4,4,4)$-face.  If $w$ is not incident to a triangle, then by Lemma~\ref{claim4.1}(3), each of the rich $4$-vertices gives at least $\frac{2}{3}$ to $f$. Thus,  $\mu^*(f)\geq-2+3\cdot \frac{2}{3}=0$. Let $w$ be incident to a triangle $f_1$. Note that $w$ is rich and none of $v$ and $x$ is poor. If $f_1$ is a $(3,4,4)$-face, then by Lemma~\ref{3-face-and-4-face} (1), each of $v$ and $x$ is incident with a triangle. In this case, by Lemma~\ref{claim4.1}(1), each of $v, w$ and $x$ gives at least $\frac{3}{4}$ to $f$. This implies that $\mu^*(f)\geq -2+3\cdot \frac{3}{4}>0$. Thus, assume that $f_1$ is not a $(3,4,4)$-face.   By (R1.1.1) and Lemma~\ref{claim4.1}(2), $w$  gives at least $1$ to $f$ and each of $v$ and $x$ gives at least $\frac{1}{2}$ to $f$. Thus, $\mu^*(f)\geq-2+1+2\cdot \frac{1}{2}=0$.

\item $f$ is a $(3,4,4,5^+)$-face. First we assume that $x$ is a $6^+$-vertex or $x$ is a rich $5$-vertex, then by Lemma~\ref{claim4.1} (2) and (4), $f$ receives at least $1$ from $x$ and $\frac{1}{2}$ from each of $v$ and $w$, thus $\mu^*(f)\ge-2+1+\frac{1}{2}\times2=0$. Now we assume that $x$ is a poor $5$-vertex. If none of the two $4$-vertices is incident to a triangle, then $f$ is a special $(3,4,4,5)$-face. By (R1.2.2) and Lemma~\ref{claim4.1} (2), $f$ receives at least $1$ from $x$ and $\frac{1}{2}$ from each of $v$ and $w$.  If both of the two $4$-vertices are incident to triangles, then by Lemma~\ref{claim4.1} (2) and (R1.2.2), $f$ gets at least $\frac{3}{4}$ from each of $v$ and $w$ and $\frac{1}{2}$ from $x$. If exactly one of the two $4$-vertices (say $v$) is incident to a triangle, then $f$ is a weak $(3,4,4,5)$-face. By (R1.2.2) and Lemma~\ref{claim4.1} (2), $f$ receives at least $\frac{3}{4}$ from each of $v$ and $x$ and $\frac{1}{2}$ from $w$.  In both cases, $\mu^*(f)\geq2-max\{1+\frac{1}{2}\cdot2,  \frac{3}{4}\cdot2+\frac{1}{2}\}=0$.

\item $f$ is a $(3,4,5^+,4)$-face. By (R1.2.2) and Lemma~\ref{claim4.1} (2) and (4), $f$ receives at least $1$ from $w$ and $\frac{1}{2}$ from each of $v$ and $x$. Thus, $\mu^*(f)\ge-2+1+\frac{1}{2}\times2=0$.

\item  $f$ is a $(3,4,5^+,5^+)$-face or $(3,5^+,4,5^+)$-face.  By (R1.2.2) and Lemma~\ref{claim4.1} (2) and (4),  $f$ receives at least $\frac{3}{4}$ from each of the two $5^+$-vertices, and $\frac{1}{2}$ from the $4$-vertex. Thus, $\mu^*(f)=-2+\frac{3}{4}\cdot2+\frac{1}{2}=0$.

\item  $f$ is a $(3,5^+,5^+,5^+)$-face.    If at least one vertex is a rich $5$-vertex or a $6^+$-vertex, then by Lemma~\ref{claim4.1}(4) and (R1.2.2),$f$ gets at least $1$ from the vertex and at least $\frac{1}{2}$ from each of the other $5$-vertices. It follows that $\mu^*(f)\ge-2+1+2\cdot\frac{1}{2}=0$.  Thus, we may assume that all are poor $5$-vertices. In this case, by (R1.2.2), $f$ is special $(3,5,5,5)$-face. Thus $f$ receives $\frac{3}{4}$ from each of the $5$-vertices.  Thus, $\mu^*(f)\ge-2+3\cdot\frac{3}{4}>0$.

\item $f$ is a $(4^+,4^+,4^+,4^+)$-face. If $f$ is rich, then $f$ contains at least two rich $5^+$ vertices or $6^+$-vertices. By Lemma~\ref{claim4.1}(4), $\mu^*(f)\geq -2+1\cdot 2=0$. Thus, we may assume that $f$ is not rich. By Lemma \ref{claim4.1}(2) each rich $4$-vertex gives at least $\frac{1}{2}$ to $f$. By Lemma \ref{claim4.1}(4) and (R1.2.2) each $5^+$-vertex gives at least $\frac{1}{2}$ to $f$. Note that each $4^+$-vertex on $f$ is not poor $4$-vertex. Thus, $f$ receives at least $\frac{1}{2}$ from each vertex on $b(f)$. So $\mu^*(f)\ge-2+4\cdot\frac{1}{2}=0$.
\end{enumerate}
\end{proof}

Now we consider the outer-face $C_0$.  Let $t_i$ be the number of $i$-vertices on $C_0$, then $d(C_0)\ge t_2+t_3+t_4$.  Note that $d(C_0)\in \{3,7\}$.  By (R3),
\begin{align*}\mu^*(C_0)&=d(C_0)+6-2t_2-\frac{3}{2}t_3-t_4\ge d(C_0)+6-\frac{3}{2}(t_2+t_3+t_4)-\frac{t_2}{2}\\
&\ge d(C_0)+6-\frac{3}{2}d(C_0)-\frac{t_2}{2}=6-\frac{d(C_0)}{2}-\frac{t_2}{2}.
\end{align*}

If $d(C_0)=3$ or $t_2\le 5$, then $\mu^*(C_0)\ge 0$. Thus, we may assume that $d(C_0)=7$ and $(t_2,t_3,t_4)\in \{(6, 1, 0), (7,0,0)\}$.  If $t_2=7$, then $G=C_0$ and it is trivially superextendable.  If $t_2=6$ and $t_3=1$, then by (R3), $C_0$ gains $1$ from the adjacent face which has degree more than $7$. Thus, $\mu^*(C_0)\ge \frac{1}{2}>0$.

We have shown that all vertices and faces have non-negative final charges.  Furthermore, the outer-face has positive charges, except when $d(C_0)=7$ and $t_2=5$ and $t_3=2$(the two $3$-vertices must be adjacent and has a common neighbor not on $C_0$) in which there must be a face other than $C_0$ having degree more than $7$. Thus the face has positive final charge.  Therefore, $\sum_{x\in V(G)\cup F(G)} \mu^*(x)>0$, a contradiction.


\begin{thebibliography}{99}

\bibitem{B12}
O. V. Borodin, Colorings of plane graphs: A survey. {\em Discrete Math.}, {\bf 313} (2013) 517--539.


\bibitem{BG04}
O. V. Borodin and A. N. Glebov, A sufficient condition for planar graphs to be $3$-colorable, {\em Diskret Anal Issled Oper. } {\bf 10} (2004) 3--11 (in Russian)


\bibitem{BG11}
O. V. Borodin and A. N. Glebov, Planar graphs with neither 5-cycles nor close 3-cycles are 3-colorable, {\em J. Graph Theory}, {\bf 66} (2011), 1--31.

\bibitem{BGRS05}
O. V. Borodin, A. N. Glebov, A. R. Raspaud, and M. R. Salavatipour. Planar graphs without cycles of length from 4 to 7 are 3-colorable. {\em J. of Combin. Theory}, Ser. B, {\bf 93} (2005), 303--311.


\bibitem{BR03}
O. V. Borodin and A. Raspaud, A sufficient condition for planar graphs to be $3$-colorable, {\em J. Combin. Theory}, Ser B, {\bf 88} (2003), 17--27.




\bibitem{CHMR11}
G. Chang, F. Havet, M. Montassier, and A. Raspaud, Steinberg's Conjecture and near colorings, manuscript.

\bibitem{DKT09}
Z. Dv\"or\'ak, D. Kr\'al and R. Thomas, Coloring planar graphs with triangles far apart, Mathematics ArXiV, arXiv:0911.0885, 2009.

\bibitem{G59}
H. Gr\"{o}tzsch,  Ein dreifarbensatz f$\:{u}$r dreikreisfreienetze auf der kugel. \emph{Math.-Nat.Reihe}, {\bf 8} (1959), 109--120.

\bibitem{H69}
I. Havel, On a conjecture of Grunbaum, {\em J. Combin. Theory}, Series B, {\bf 7} (1969) 184--186.



\bibitem{HSWXY13}
O. Hill, D. Smith, Y. Wang, L. Xu, and G. Yu, Planar graphs without $4$-cycles and $5$-cycles are $(3,0,0)$-colorable, {\em Discrete Math.}, {\bf 313} (2013) 2312--2317.

\bibitem{HY13}
O. Hill, and G. Yu,  A relaxation of Steinberg's Conjecture, {\em SIAM J. of Discrete Math.},  {\bf 27} (2013) 584--596.

\bibitem{LLY14}
R. Liu, X. Li, and G. Yu, A relaxation of the Bordeaux Conjecture, http://arxiv.org/abs/1407.5138.  Submitted.

\bibitem{S76}
R. Steinberg,  The state of the three color problem. Quo Vadis, Graph Theory?, {\em Ann. Discrete
Math.} {\bf 55} (1993), 211--248.

\bibitem{X07}
B. Xu, A $3$-color theorem on plane graph without 5-circuits, {\em Acta Math Sinica},{\bf  23}  (2007) 1059--1062.


\bibitem{X08}
B. Xu, On $(3,1)^*$-coloring of planar graphs, {\em SIAM J. Disceret Math.}, {\bf 23} (2008), 205--220.

\bibitem{XMW12}
L. Xu, Z. Miao, and Y. Wang, Every planar graph with cycles of length neither $4$ nor $5$ is $(1,1,0)$-colorable, {\em J Comb. Optim.}, DOI 10.1007/s10878-012-9586-4.

\bibitem{XW13}
L. Xu  and Y. Wang,  Improper colorability of planar graphs with cycles of length neither 4 nor 6 (in Chinese), {\em Sci Sin Math}, {\bf 43} (2013), 15-24.

\bibitem{YY14}
 C. Yang, and C. Yerger,  The Bordeaux 3-color Conjecture and Near-Coloring, preprint.



\end{thebibliography}
\end{document}